\documentclass{article}
 \usepackage[final, nonatbib]{neurips_2019}
\usepackage[utf8]{inputenc} 
\usepackage[T1]{fontenc}    
\usepackage{hyperref}       
\usepackage{url}            
\usepackage{booktabs}       
\usepackage{amsfonts}       
\usepackage{nicefrac}       
\usepackage{microtype}      
\usepackage{amsmath}
\usepackage{graphicx}
\usepackage{multirow}
\usepackage{amsthm}

\usepackage{algorithm}
\usepackage{algorithmic}
\usepackage{color}
\usepackage{enumitem}
\usepackage{comment}

\usepackage{hyperref}
\newcommand{\RR}{\mathbb{R}}

\newcommand{\ol}{\overline}

\newcommand{\s}{\sigma}

\renewcommand{\b}{\beta}
\renewcommand{\L}{\mathcal{L}} 

\newcommand{\argmin}{\operatorname{argmin}}
\newcommand{\dist}{\operatorname{dist}}

\newcommand{\subjto}{\quad \text{s.t.} \quad}

\newtheorem{theorem}{Theorem}[section] 

\newtheorem{lemma}[theorem]{Lemma}
\newtheorem{corollary}[theorem]{Corollary}

\title{An  Inexact Augmented Lagrangian Framework for \\ Nonconvex Optimization with Nonlinear Constraints}

\author{Mehmet Fatih Sahin\\ \texttt{mehmet.sahin@epfl.ch}
\And Armin Eftekhari \\ \texttt{armin.eftekhari@epfl.ch}
\And Ahmet Alacaoglu \\ \texttt{ahmet.alacaoglu@epfl.ch}
\And Fabian Latorre \\ \texttt{fabian.latorre@epfl.ch}
\And Volkan Cevher \\ \texttt{volkan.cevher@epfl.ch} 
\And \\
LIONS, Ecole Polytechnique F\'ed\'erale de Lausanne, Switzerland
}

\begin{document}

\maketitle


\begin{abstract}
We propose a practical inexact augmented Lagrangian method (iALM) for nonconvex problems with nonlinear constraints. 
We characterize the total computational complexity of our method subject to a verifiable geometric condition, which is closely related to the Polyak-Lojasiewicz and Mangasarian-Fromowitz conditions. 

\vspace{1mm}
In particular, when a first-order solver is used for the inner iterates, we prove that iALM  finds a first-order stationary point with {$\tilde{\mathcal{O}}(1/\epsilon^4)$} calls to the first-order oracle. {If, in addition, the problem is smooth and} a second-order solver is used for the inner iterates, iALM  finds a second-order stationary point with $\tilde{\mathcal{O}}(1/\epsilon^5)$ calls to the second-order oracle, {which matches the known theoretical complexity result in the literature.} 

\vspace{1mm}
We also provide strong numerical evidence on large-scale machine learning problems, including the Burer-Monteiro factorization of  semidefinite programs, and a novel nonconvex relaxation of the standard basis pursuit template. For these examples, we also show how to verify our geometric condition.

{\color{blue} This paper was updated on April 21, 2022, due to an error in the proof of Corollary \eqref{cor:first}. The gradient complexity for obtaining a first order stationary point is now $\tilde{\mathcal{O}}( 1/{\epsilon^4})$ instead of $\tilde{\mathcal{O}}( 1/{\epsilon^3})$. The error was due to an order mistake while using the complexity result of APGM from \cite{ghadimi2016accelerated}}.

\end{abstract}

\vspace{-5mm}
\section{Introduction}\vspace{-3mm}
\label{intro}
We study the  nonconvex optimization problem
\begin{equation}
\label{prob:01}
\underset{x\in \mathbb{R}^d}{\min}\,\, f(x)+g(x)  \subjto  A(x) = 0,
\end{equation}
where $f:\mathbb{R}^d\rightarrow\mathbb{R}$ is  a {continuously-differentiable} nonconvex function and $A:\mathbb{R}^d\rightarrow\mathbb{R}^m$ is a  nonlinear operator. 
We assume that $g:\mathbb{R}^d\rightarrow\mathbb{R} \cup \{\infty\}$ is a proximal-friendly convex function \cite{parikh2014proximal}. 


A host of problems in computer science \cite{khot2011grothendieck, lovasz2003semidefinite, zhao1998semidefinite}, machine learning \cite{mossel2015consistency, song2007dependence}, and signal processing~\cite{singer2011angular, singer2011three} naturally fall under the template~\eqref{prob:01},  including max-cut, clustering, generalized eigenvalue decomposition, as well as the {quadratic assignment problem (QAP) \cite{zhao1998semidefinite}}.

To solve \eqref{prob:01}, we 
propose an intuitive and easy-to-implement {augmented Lagrangian} algorithm, and provide its total iteration complexity under an interpretable geometric condition. Before we elaborate on the results, let us first motivate~\eqref{prob:01} with an application to semidefinite programming (SDP): 

\paragraph{Vignette: Burer-Monteiro splitting.}
A powerful convex relaxation for max-cut, clustering, and many others is provided by the trace-constrained SDP 
\begin{equation}
\label{eq:sdp}
\underset{X\in\mathbb{S}^{d \times d}}{\min} \langle C, X \rangle  \subjto  B(X) = b, \,\, \text{tr}(X)\leq \alpha, \,\, X \succeq 0,
\end{equation}
where $C\in \RR^{d\times d}$, $X$ is a positive semidefinite $d\times d$ matrix, 
and ${B}: \mathbb{S}^{d\times d} \to \mathbb{R}^m$ is a linear operator. If the unique-games conjecture is true, the SDP \eqref{eq:sdp} obtains the best possible approximation for the underlying discrete problem~\cite{raghavendra2008optimal}.

Since $d$ is often large, many first- and second-order methods for solving such SDP's are immediately ruled out, not only due to their high computational complexity, but also due to their storage requirements, which are $\mathcal{O}(d^2)$. 

A contemporary challenge in optimization is therefore to solve SDPs using little space and in a scalable fashion. The recent homotopy conditional gradient method, which is based on linear minimization oracles (LMOs), can solve \eqref{eq:sdp} in a small space via sketching \cite{yurtsever2018conditional}. However, such LMO-based methods are extremely slow in obtaining accurate solutions.


A different approach for solving \eqref{eq:sdp}, dating back to~\cite{burer2003nonlinear, burer2005local}, is the so-called Burer-Monteiro (BM) factorization $X=UU^\top$, where $U\in\mathbb{R}^{d\times r}$ and $r$ is selected according to the guidelines in~\cite{pataki1998rank, barvinok1995problems}, which is tight~\cite{waldspurger2018rank}. 
The BM factorization leads to the following  nonconvex problem in the template~\eqref{prob:01}:
\begin{equation}
\label{prob:nc}
\underset{U\in\mathbb{R}^{d \times r}}{\min} \langle C, UU^\top \rangle \subjto B(UU^\top) = b, \,\, \| U\|_F^2 \leq \alpha,
\end{equation}
The BM factorization does not introduce any extraneous local minima~\cite{burer2005local}. Moreover,~\cite{boumal2016non} establishes the connection between the local minimizers of the factorized problem~\eqref{prob:nc} and the global minimizers for~\eqref{eq:sdp}. 
%
%
To solve \eqref{prob:nc}, the inexact Augmented Lagrangian method (iALM) is widely used~\cite{burer2003nonlinear, burer2005local, kulis2007fast}, due to its cheap per iteration cost and its empirical success. 

Every (outer) iteration of iALM calls a solver to solve an intermediate augmented Lagrangian subproblem to near stationarity. The choices include first-order methods, such as the proximal gradient descent \cite{parikh2014proximal}, or second-order methods, such as the trust region method and BFGS~\cite{nocedal2006numerical}.\footnote{BFGS is in fact a quasi-Newton method that emulates second-order information.}

Unlike its convex counterpart~\cite{nedelcu2014computational,lan2016iteration,xu2017inexact}, the convergence rate and the complexity of iALM for ~\eqref{prob:nc} are not well-understood, see Section~\ref{sec:related work} for a review of the related literature. Indeed, addressing this important theoretical gap is one of the contributions of our work. In addition,
{

$\triangleright$ {
We derive the convergence rate of iALM 
to first-order optimality for solving \eqref{prob:01}  or second-order optimality for solving \eqref{prob:01} with $g=0$, and find the total iteration complexity of iALM using different solvers for the augmented Lagrangian subproblems. 
{We provide an extensive comparison with the existing complexity results in optimization, see Section~\ref{sec:related work}}.}  \\[2mm]
$\triangleright$ Our iALM framework is future-proof in the sense that different subsolvers can be substituted. \\[2mm]
$\triangleright$ We propose a geometric condition that simplifies the algorithmic analysis for iALM, and clarify its connection to well-known Polyak-Lojasiewicz \cite{karimi2016linear} and Mangasarian-Fromovitz \cite{bertsekas1982constrained} conditions. 
We also verify this condition for key problems in Appendices \ref{sec:verification2} and \ref{sec:adexp}. 


\vspace{-4mm}
\section{Preliminaries \label{sec:preliminaries}}\vspace{-3mm}
\paragraph{\textbf{{Notation.}}}
We use the notation $\langle\cdot ,\cdot \rangle $ and $\|\cdot\|$ for the {standard inner} product and the norm on $\RR^d$. For matrices, $\|\cdot\|$ and $\|\cdot\|_F$ denote the spectral and the Frobenius norms, respectively.  
For the convex function $g:\RR^d\rightarrow\RR$, the subdifferential set at $x\in \RR^d$ is denoted by $\partial g(x)$ and  we will occasionally use the notation $\partial g(x)/\b = \{ z/\b : z\in \partial g(x)\}$. 
When presenting iteration complexity results, we often use $\widetilde{O}(\cdot)$ which suppresses the logarithmic dependencies.

We denote $\delta_\mathcal{X}:\RR^d\rightarrow\RR$ as the indicator function  of a set $\mathcal{X}\subset\RR^d$. 
%
The distance function from a point $x$ to $\mathcal{X}$ is denoted by $\dist(x,\mathcal{X}) = \min_{z\in \mathcal{X}} \|x-z\|$. 
For integers $k_0 \le k_1$, we use the  notation $[k_0:k_1]=\{k_0,\ldots,k_1\}$. For an operator $A:\RR^d\rightarrow\RR^m$ with components $\{A_i\}_{i=1}^m$,  $DA(x) \in \mathbb{R}^{m\times d}$ denotes the Jacobian of $A$, where the $i$th row of $DA(x)$ is the  vector $\nabla A_i(x) \in \RR^d$. 


\paragraph{Smoothness.}
We assume smooth $f:\RR^d\rightarrow\RR$ and $A:\RR^d\rightarrow \RR^m$; i.e., there exist $\lambda_f,\lambda_A\ge 0$ s.t.
%
\begin{align}
\| \nabla f(x) - \nabla f(x')\|  \le \lambda_f \|x-x'\|, \quad  \| DA(x) - DA(x') \|  \le \lambda_A \|x-x'\|,
\quad  \forall x, x' \in \RR^d .
\label{eq:smoothness basic}
\end{align}

\paragraph{Augmented Lagrangian method (ALM).}
ALM is a classical algorithm, which first appeared in~\cite{hestenes1969multiplier, powell1969method} and extensively studied afterwards in~\cite{bertsekas1982constrained, birgin2014practical}.
For solving \eqref{prob:01}, ALM suggests solving the problem
%
\begin{equation}
\min_{x} \max_y \,\,\mathcal{L}_\beta(x,y) + g(x), 
\label{eq:minmax}
\end{equation}
%
where, for penalty weight $\b>0$,  $\mathcal{L}_\b$ is the corresponding augmented Lagrangian, defined as 
\begin{align}
\label{eq:Lagrangian}
\mathcal{L}_\beta(x,y) := f(x) + \langle A(x), y \rangle +  \frac{\beta}{2}\|A(x)\|^2.
\end{align}
The minimax formulation in \eqref{eq:minmax} naturally suggests the following algorithm for solving \eqref{prob:01}:
\begin{equation}\label{e:exac}
x_{k+1}  \in \underset{x}{\argmin} \,\, \mathcal{L}_{\beta}(x,y_k)+g(x), 
\end{equation}
\begin{equation*}
y_{k+1}   = y_k+\s_k A(x_{k+1}), 
\end{equation*}
where the  dual step sizes  are denoted as $\{\s_k\}_k$.
However, computing $x_{k+1}$ above requires solving the nonconvex problem~\eqref{e:exac} to optimality, which is typically intractable. Instead, it is  often easier to find an approximate first- or second-order stationary point of~\eqref{e:exac}.

Hence, we argue that by gradually improving the stationarity precision and  increasing the penalty weight $\b$ above, we can reach a stationary point of the main problem in~\eqref{eq:minmax}, as detailed  in Section~\ref{sec:AL algorithm}. 

\paragraph{{\textbf{Optimality conditions.}}}
{First-order necessary optimality conditions} for \eqref{prob:01} are well-studied. {Indeed, $x\in \RR^d$ is a first-order stationary point of~\eqref{prob:01} if there exists $y\in \RR^m$ such that 
%
%
\begin{align}
-\nabla_x \mathcal{L}_\beta(x,y) \in \partial g(x),\qquad  A(x) = 0,
\label{e:inclu2}
\end{align}
which is in turn the necessary optimality condition for \eqref{eq:minmax}.} 
Inspired by this, we say that $x$ is an $(\epsilon_f,\b)$ first-order stationary point of \eqref{eq:minmax} if {there exists a $y \in \RR^m$} such that 
\begin{align}
\dist(-\nabla_x \mathcal{L}_\beta(x,y), \partial g(x)) \leq \epsilon_f, \qquad  \| A(x) \| \leq \epsilon_f,
\label{eq:inclu3}
\end{align}
for $\epsilon_f\ge 0$. 
In light of \eqref{eq:inclu3}, a metric for evaluating the stationarity of a pair $(x,y)\in \RR^d\times \RR^m$ is 
\begin{align}
\dist\left(-\nabla_x \mathcal{L}_\beta(x,y), \partial g(x) \right) + \|A(x)\| ,
\label{eq:cvg metric}
\end{align}
which we use as the first-order stopping criterion. 
As an example, for a convex set $\mathcal{X}\subset\RR^d$, suppose that $g = \delta_\mathcal{X}$ is the indicator function on $\mathcal{X}$.
Let also $T_\mathcal{X}(x) \subseteq \RR^d$ denote the tangent cone to $\mathcal{X}$ at $x$, and with $P_{T_\mathcal{X}(x)}:\RR^d\rightarrow\RR^d$ we denote the orthogonal projection onto this tangent cone. Then, for $u\in\RR^d$, it is not difficult to verify that 
\begin{align}\label{eq:dist_subgrad}
\dist\left(u, \partial g(x) \right) = \| P_{T_\mathcal{X}(x)}(u) \|.
\end{align} 
When $g = 0$, a first-order stationary point $x\in \RR^d$ of \eqref{prob:01}   is also second-order stationary if 
%
\begin{equation}
\lambda_{\text{min}}(\nabla _{xx} \mathcal{L}_{\beta}(x,y))\ge 0,
\end{equation}
%
where $\nabla_{xx}\mathcal{L}_\b$ is the Hessian of $\mathcal{L}_\b$ with respect to $x$, and $\lambda_{\text{min}}(\cdot)$ returns the smallest eigenvalue of its argument.
Analogously, $x$ is an $(\epsilon_f, \epsilon_s,\b)$ second-order stationary point if, in addition to \eqref{eq:inclu3}, it holds that  
%
\begin{equation}\label{eq:sec_opt}
\lambda_{\text{min}}(\nabla _{xx} \mathcal{L}_{\beta}(x,y)) \ge -\epsilon_s,
\end{equation}
for  $\epsilon_s\ge 0$. 
%
Naturally, for second-order stationarity, we  use $\lambda_{\text{min}}(\nabla _{xx} \mathcal{L}_{\beta}(x,y))$ as the stopping criterion.

%


\paragraph{{\textbf{Smoothness lemma.}}} This next result controls the smoothness of $\L_\b(\cdot,y)$ for a fixed $y$. The proof is standard but nevertheless is included in Appendix~\ref{sec:proof of smoothness lemma} for completeness. 
\begin{lemma}[\textbf{smoothness}]\label{lem:smoothness}
 For fixed $y\in \RR^m$ and $\rho,\rho'\ge 0$, it holds that 
 \begin{align}
\| \nabla_x \mathcal{L}_{\beta}(x, y)-  \nabla_x \mathcal{L}_{\beta}(x', y) \| \le \lambda_\b \|x-x'\|,
 \end{align}
 for every $x,x' \in \{ x'': \|x''\|\le \rho, \|A(x'') \|\le \rho'\}$, where
 \begin{align}
\lambda_\beta 
 \le \lambda_f + \sqrt{m}\lambda_A \|y\| +  (\sqrt{m}\lambda_A\rho'  +  d \lambda'^2_A )\b
 =: \lambda_f + \sqrt{m}\lambda_A \|y\| +   \lambda''(A,\rho,\rho') \b.
\label{eq:smoothness of Lagrangian}
\end{align}
Above, $\lambda_f,\lambda_A$ were defined in (\ref{eq:smoothness basic}) and 
\begin{align}
\lambda'_A := \max_{\|x\|\le \rho}\|DA(x)\|.
\end{align}
\end{lemma}

\vspace{-4mm}
\section{Algorithm \label{sec:AL algorithm}}

To solve the equivalent formulation of \eqref{prob:01}  presented in \eqref{eq:minmax}, we  propose the inexact ALM (iALM), detailed in Algorithm~\ref{Algo:2}. 
At the $k^{\text{th}}$ iteration, Step 2 of Algorithm~\ref{Algo:2} calls a solver that finds an approximate stationary point of the augmented Lagrangian $\L_{\b_k}(\cdot,y_k)$ with the accuracy of $\epsilon_{k+1}$, and this accuracy gradually increases in a controlled fashion. 
The increasing sequence of penalty weights $\{\b_k\}_k$ and the dual update (Steps 4 and 5) are responsible for continuously enforcing the constraints in~\eqref{prob:01}. The appropriate choice for $\{\b_k\}_k$ will be specified in Corrollary  Sections  \ref{sec:first-o-opt} and \ref{sec:second-o-opt}. 

The particular choice of the dual step sizes $\{\s_k\}_k$ in Algorithm~\ref{Algo:2} ensures that the dual variable $y_k$ remains bounded.


\begin{algorithm}[h!]
\begin{algorithmic}
\STATE \textbf{Input:} Non-decreasing, positive, unbounded sequence $\{\b_k\}_{k\ge 1}$, stopping thresholds $\tau_f, \tau_s  > 0$.  \vspace{2pt}
\STATE \textbf{Initialization:} Primal variable $x_{1}\in \RR^d$, dual variable $y_0\in \RR^m$, dual step size $\s_1>0$. 
\vspace{2pt}
\FOR{$k=1,2,\dots$}
\STATE \begin{enumerate}
\item \textbf{(Update tolerance)}  $\epsilon_{k+1} = 1/\b_k$. 
\item  \textbf{(Inexact primal solution)} Obtain $x_{k+1}\in \RR^d$ such that 
\begin{equation*}
\dist(-\nabla_x \L_{\beta_k} (x_{k+1},y_k), \partial g(x_{k+1}) )  \le \epsilon_{k+1}
\end{equation*}
 for first-order stationarity
\begin{equation*}
\lambda_{\text{min}}(\nabla _{xx}\mathcal{L}_{\beta_k}(x_{k+1}, y_k)) \ge  -\epsilon_{k+1}
\end{equation*}
for second-order-stationarity, if $g=0$ in \eqref{prob:01}. 
\item \textbf{(Update dual step size)} 
\begin{align*}
\s_{k+1} & = \s_{1} \min\Big( 
\frac{\|A(x_1)\| \log^2 2 }{\|A(x_{k+1})\| (k+1)\log^2(k+2)} ,1
\Big). 
\end{align*}
\item \textbf{(Dual ascent)}  $y_{k+1} = y_{k} + \sigma_{k+1}A(x_{k+1})$.
\item \textbf{(Stopping criterion)} If 
\begin{align*}
& \dist(-\nabla_x \L_{\b_k}(x_{k+1}),\partial g(x_{k+1}))  +  \|A(x_{k+1})\| \le \tau_f,\nonumber
\end{align*}
for first-order stationarity and if also 
$\lambda_{\text{min}}(\nabla _{xx}\mathcal{L}_{\beta_{k}}(x_{k+1}, y_k)) \geq -\tau_s$ for second-order stationarity, 
 then quit and return $x_{k+1}$ as an (approximate) stationary point of \eqref{eq:minmax}. 
\end{enumerate}
 \ENDFOR
\end{algorithmic}
\caption{Inexact ALM}
\label{Algo:2}
\end{algorithm}


\section{Convergence Rate \label{sec:cvg rate}}
This section presents the total iteration complexity of Algorithm~\ref{Algo:2} for finding first and second-order stationary points of problem \eqref{eq:minmax}. 
All the proofs are deferred to Appendix~\ref{sec:theory}.
Theorem~\ref{thm:main} characterizes the convergence rate of Algorithm~\ref{Algo:2} for finding stationary points in the number of outer iterations.
\begin{theorem}\textbf{\emph{(convergence rate)}}
\label{thm:main}  For integers $2 \le k_0\le k_1$, consider the interval $K=[k_0:k_1]$, and let $\{x_k\}_{k\in K}$ be the output sequence of Algorithm~\ref{Algo:2} on the interval $K$.\footnote{The choice of $k_1 = \infty$ is valid here too.}  Let also $\rho:=\sup_{k\in [K]} \|x_k\|$.\footnote{If necessary, to ensure that $\rho<\infty$, one can add a small factor of $\|x\|^2$ to $\mathcal{L}_{\b}$ in \eqref{eq:Lagrangian}. Then it is easy to verify that the iterates of Algorithm \ref{Algo:2} remain bounded, provided that the initial penalty weight $\beta_0$ is large enough, $\sup_x \|\nabla f(x)\|/\|x\|< \infty$,  $\sup_x \|A(x)\| < \infty$, and $\sup_x \|D A(x)\| <\infty$.  } Suppose that $f$ and $A$ satisfy (\ref{eq:smoothness basic}) and  let 
\begin{align}
\lambda'_f = \max_{\|x\|\le \rho} \|\nabla f(x)\|,\qquad 
\lambda'_A = \max_{\|x\| \le \rho} \|DA(x)\|,
\label{eq:defn_restricted_lipsichtz}
\end{align}
be the (restricted) Lipschitz constants of $f$ and $A$, respectively. 
With $\nu>0$, assume that 
 \begin{align}
\nu \|A(x_k)\|  
& \le \dist\left( -DA(x_k)^\top A(x_k) , \frac{\partial g(x_k)}{ \b_{k-1}}  \right), 
\label{eq:regularity}
\end{align}
for every $k\in K$. We consider two cases:
\begin{itemize}[leftmargin=*]
\item If a first-order solver is used in Step~2, then $x_k$ is an $(\epsilon_{k,f},\b_k)$ first-order stationary point of (\ref{eq:minmax}) with 
\begin{align}
\epsilon_{k,f} & =  \frac{1}{\beta_{k-1}} \left(\frac{2(\lambda'_f+\lambda'_A y_{\max}) (1+\lambda_A' \sigma_k)}{\nu}+1\right) =: \frac{Q(f,g,A,\s_1)}{\beta_{k-1}}, 
\label{eq:stat_prec_first}
\end{align}
for every $k\in K$, where $y_{\max}(x_1,y_0,\s_1):=\|y_0\|+ c  \|A(x_1) \|$.
\item If a second-order solver is used in Step~2, then $x_k$ is an $(\epsilon_{k,f}, \epsilon_{k,s},\b_k)$ second-order stationary point of~(\ref{eq:minmax}) with $\epsilon_{k,s}$ specified above and with 
\begin{align}
\epsilon_{k,s} &= \epsilon_{k-1} + \sigma_k \sqrt{m} \lambda_A \frac{ 2\lambda'_f +2 \lambda'_A y_{\max} }{\nu \b_{k-1}} = \frac{\nu + \sigma_k \sqrt{m} \lambda_A 2\lambda'_f +2 \lambda'_A y_{\max} }{\nu \b_{k-1}}  =: \frac{Q'(f,g,A,\s_1)}{\beta_{k-1}}.
\end{align}
\end{itemize}
\end{theorem}



Theorem~\ref{thm:main} states that Algorithm~\ref{Algo:2} converges to a (first- or second-) order stationary point of \eqref{eq:minmax} at the rate of $1/\b_k$, further specified in 
Corollary \ref{cor:first} and Corollary \ref{cor:second}.
A few remarks are in order about Theorem \ref{thm:main}.

\paragraph{Regularity.}  The key geometric condition in Theorem~\ref{thm:main} is \eqref{eq:regularity} which, broadly speaking, ensures that the primal updates of Algorithm \ref{Algo:2} reduce the feasibility gap as the penalty weight $\b_k$ grows. We will verify this condition for several examples in Appendices \ref{sec:verification2} and \ref{sec:adexp}.

This condition in \eqref{eq:regularity} is closely related to those in the existing literature. In the special case where $g=0$ in~\eqref{prob:01}, \eqref{eq:regularity} reduces to;
\begin{align}
\|DA(x)^\top A(x)\| \geq \nu \|A(x)\|.
\label{eq:regularity_special}
\end{align}
\emph{Polyak-Lojasiewicz (PL) condition~\cite{karimi2016linear}.} Consider the  problem with $\lambda_{\tilde{f}}$-smooth objective,
\begin{align*}
\min_{x\in\RR^d} \tilde{f}(x).
\end{align*}
$\tilde{f}(x)$ satisfies the PL inequality if the following holds for some $\mu > 0$,
\begin{equation}
\frac{1}{2} \| \nabla \tilde{f}(x) \|^2 \geq \mu (\tilde{f}(x) - \tilde{f}^*),\quad \forall x
\tag{PL inequality}
\end{equation}
This inequality implies that gradient is growing faster than a quadratic as we move away from the optimal. Assuming that the feasible set $\{ x: A(x) = 0\}$ is non-empty, it is easy to verify that \ref{eq:regularity_special} is equivalent to 
 the PL condition for minimizing $\tilde{f}(x) = \frac{1}{2}\|A(x)\|^2$ with $\nu = \sqrt{2 \mu} $~~\cite{karimi2016linear}. 

PL condition itself is a special case of Kurdyka-Lojasiewicz with $\theta = 1/2$, see \cite[Definition 1.1]{xu2017globally}. When $g=0$, it is also easy to see that \eqref{eq:regularity} is  weaker than the Mangasarian-Fromovitz (MF) condition in nonlinear optimization \cite[Assumption 1]{bolte2018nonconvex}. 
{Moreover, {when $g$ is the indicator on a convex set,} \eqref{eq:regularity} is a consequence of the \textit{basic constraint qualification} in \cite{rockafellar1993lagrange}, which itself generalizes the MF condition to the case when $g$ is an indicator function of a convex set.} 

We may think of  \eqref{eq:regularity} as a local condition, which should hold within a neighborhood of the constraint set $\{x:A(x)=0\}$ rather than everywhere in $\mathbb{R}^d$. 
Indeed, the iteration count $k$ appears in \eqref{eq:regularity} to reflect this local nature of the condition. 
Similar kind of arguments on the regularity condition also appear in \cite{bolte2018nonconvex}. 
There is also a constant complexity algorithm in \cite{bolte2018nonconvex} to reach so-called ``information zone'', which supplements Theorem \ref{thm:main}. 
\paragraph{Penalty method.}
A classical algorithm to solve \eqref{prob:01} is the penalty method, which is characterized by the absence of the dual variable ($y=0$) in \eqref{eq:Lagrangian}. Indeed, ALM can be interpreted as an adaptive penalty or smoothing method with a variable center  determined by the dual variable. It is worth noting that, with the same proof technique, one can establish the same convergence rate of Theorem \ref{thm:main} for the penalty method. However, while both methods have the same convergence rate in theory, we ignore the uncompetitive penalty method since it is significantly outperformed by iALM in practice. 

\paragraph{Computational complexity.} Theorem~\ref{thm:main} specifies the number of (outer) iterations that Algorithm~\ref{Algo:2} requires to reach a near-stationary point of problem~\eqref{eq:Lagrangian} with a  prescribed precision and, in particular, specifies the number of calls made to the solver in Step~2.  In this sense, Theorem~\ref{thm:main} does not fully capture the computational complexity of Algorithm~\ref{Algo:2}, as it does not take into account the computational cost of the solver in Step~2. 

To better understand the total iteration complexity of Algorithm~\ref{Algo:2}, we consider two scenarios in the following. In the first scenario, we take the solver in Step~2 to be the Accelerated Proximal Gradient Method (APGM), a well-known first-order algorithm~\cite{ghadimi2016accelerated}. In the second scenario, we will use the second-order trust region method developed in~\cite{cartis2012complexity}.
We have the following two corollaries showing the total complexity of our algorithm to reach first and second-order stationary points. 
Appendix \ref{sec:opt_cnds} contains the proofs and more detailed discussion for the complexity results.

{
\begin{corollary}[First-order optimality]\label{cor:first}
For $b>1$, let $\beta_k =b^k $ for every $k$. If we use APGM from~\cite{ghadimi2016accelerated} for Step~2 of Algorithm~\ref{Algo:2}, the algorithm finds an $(\epsilon_f,\b_k)$ first-order stationary point of~\eqref{eq:minmax}, 
 after $T$ calls to the first-order oracle, where
\begin{equation}
T = \mathcal{O}\left( \frac{Q^3 \rho^2}{\epsilon^{4}}\log_b{\left( \frac{Q}{\epsilon} \right)} \right) =   \tilde{\mathcal{O}}\left( \frac{Q^{3} \rho^2}{\epsilon^{4}} \right). 
\end{equation}
\end{corollary}

}
For Algorithm~\ref{Algo:2} to reach a near-stationary point with an accuracy of $\epsilon_f$ in the sense of \eqref{eq:inclu3} and with the lowest computational cost, we therefore need to perform only one iteration of Algorithm~\ref{Algo:2}, with $\b_1$ specified as a function of $\epsilon_f$ by \eqref{eq:stat_prec_first} in Theorem~\ref{thm:main}. In general, however, the constants in \eqref{eq:stat_prec_first} are unknown and this approach is thus not feasible. Instead, the homotopy approach taken by Algorithm~\ref{Algo:2} ensures achieving the desired accuracy by gradually increasing the penalty weight. This homotopy approach increases the computational cost of Algorithm~\ref{Algo:2} only by a  factor logarithmic in the $\epsilon_f$, as detailed in the proof of Corollary~\ref{cor:first}.

\begin{corollary}[Second-order optimality]\label{cor:second}
For $b>1$, let $\beta_k =b^k $ for every $k$.
We assume that 
\begin{equation}
\mathcal{L}_{\beta}(x_1, y) - \min_{x}\mathcal{L}_{\beta}(x, y) \leq L_{u},\qquad  \forall \beta.
\end{equation}
If we use the trust region method from~\cite{cartis2012complexity} for Step~2 of Algorithm~\ref{Algo:2}, the algorithm finds an $\epsilon$-second-order stationary point of~\eqref{eq:minmax} in $T$ calls to the second-order oracle where
\begin{equation}
T = \mathcal{O}\left( \frac{L_u Q'^{5}}{\epsilon^{5}} \log_b{\left( \frac{Q'}{\epsilon} \right)} \right) = \widetilde{\mathcal{O}}\left( \frac{L_u Q'^{5}}{\epsilon^{5}} \right).
\end{equation}
\end{corollary}

\paragraph{Remark.} 
These complexity results for first and second-order are stationarity with respect to~\eqref{eq:Lagrangian}.
{We note that second order complexity result matches~\cite{cartis2018optimality} and~\cite{birgin2016evaluation}.}
However, the stationarity criteria and the definition of dual variable in these papers differ from ours.
We include more discussion on this in the Appendix.

\paragraph{Effect of  $\beta_k$ in \ref{eq:regularity}.}
 We consider two cases, when $g$ is the indicator of a convex set (or 0), the subdifferential set will be a cone (or 0), thus $\beta_k$ will not have an effect. On the other hand, when $g$ is a convex and Lipschitz contiunous function defined on the whole space, subdifferential set will be bounded \cite[Theorem 23.4]{rockafellar1970convex}. This will introduce an error term in  \ref{eq:regularity} that is of the order (1/$\beta_k$). One can see that $b^k$ choice for $\beta_k$ causes a linear decrease in this error term. In fact, all the examples in this paper fall into the first case.


\section{Related Work \label{sec:related work}} 

{{
ALM has a long history in the optimization literature, dating back to~\cite{hestenes1969multiplier, powell1969method}.
In the special case of~\eqref{prob:01} with a convex function $f$ and a linear operator $A$, standard,  inexact, and linearized versions of ALM have been extensively studied~\cite{lan2016iteration,nedelcu2014computational,tran2018adaptive,xu2017inexact}.

Classical works on ALM focused on the general template of~\eqref{prob:01} with nonconvex $f$ and nonlinear $A$, with arguably stronger assumptions and required exact solutions to  the subproblems of the form \eqref{e:exac}, which appear in Step 2 of Algorithm~\ref{Algo:2}, see for instance  \cite{bertsekas2014constrained}.

A similar analysis was conducted in~\cite{fernandez2012local} for the general template of~\eqref{prob:01}.
The authors considered inexact ALM and proved convergence rates for the outer iterates, under specific assumptions on the initialization of the dual variable. 
However, in contrast, the authors did not analyze how to solve the subproblems inexactly and  did not provide total complexity results with verifiable conditions.

Problem~\eqref{prob:01} with similar assumptions to us is also studied in~\cite{birgin2016evaluation} and~\cite{cartis2018optimality} for first-order and second-order stationarity, respectively, with explicit iteration complexity analysis.
{As we have mentioned in Section~\ref{sec:cvg rate}, our second order iteration complexity result matches these theoretical algorithms with a simpler algorithm and a simpler analysis.}
In addition, these algorithms require setting final accuracies since they utilize this information in the algorithm while our Algorithm~\ref{Algo:2} does not set accuracies a priori.

\cite{cartis2011evaluation} also considers the same template~\eqref{prob:01} for first-order stationarity with a penalty-type method instead of ALM.
Even though the authors show $\mathcal{O}(1/\epsilon^2)$ complexity, this result is obtained by assuming that the penalty parameter remains bounded. We note that such an assumption can also be used to improve our complexity results to match theirs.

\cite{bolte2018nonconvex} studies the general template~\eqref{prob:01} with specific assumptions involving local error bound conditions for the~\eqref{prob:01}.
These conditions are studied in detail in~\cite{bolte2017error}, but their validity for general SDPs~\eqref{eq:sdp} has never been established. This work also lacks the total iteration complexity analysis presented here. 

Another work~\cite{clason2018acceleration} focused on solving~\eqref{prob:01} by adapting the primal-dual method of Chambolle and Pock~\cite{chambolle2011first}.
The authors proved the convergence of the method and provided convergence rate by imposing error bound conditions on the objective function that do not hold for standard SDPs.

\cite{burer2003nonlinear, burer2005local} is the first work that proposes the splitting $X=UU^\top$ for solving SDPs of the form~\eqref{eq:sdp}.
Following these works, the literature on Burer-Monteiro (BM) splitting for the large part focused on using ALM for solving the reformulated problem~\eqref{prob:nc}.
However, this proposal has a few drawbacks: First, it requires exact solutions in Step 2 of Algorithm~\ref{Algo:2} in theory, which in practice is replaced with inexact solutions. Second, their results only establish convergence without providing the rates.  In this sense, our work provides a theoretical understanding of the BM splitting with inexact solutions to Step 2 of Algorithm~\ref{Algo:2} and complete iteration complexities.

\cite{bhojanapalli2016dropping, park2016provable} are among the earliest efforts to show convergence rates for BM splitting, focusing on the special case of SDPs without any linear constraints.
For these specific problems, they prove the convergence of gradient descent to global optima with convergence rates, assuming favorable initialization.
These results, however, do not apply to general SDPs of the form~\eqref{eq:sdp} where the difficulty arises due to the linear constraints.

Another popular method for solving SDPs are due to~\cite{boumal2014manopt, boumal2016global, boumal2016non}, focusing on the case where the constraints in~\eqref{prob:01} can be written as a  Riemannian manifold after BM splitting.
In this case, the authors apply the Riemannian gradient descent and Riemannian trust region methods for obtaining first- and second-order stationary points, respectively. 
They obtain~$\mathcal{O}(1/\epsilon^2)$ complexity for finding first-order stationary points  and~$\mathcal{O}(1/\epsilon^3)$ complexity for finding second-order stationary points. 

While these complexities appear better than ours, the smooth manifold requirement in these works is indeed restrictive. In particular, this requirement holds for max-cut and generalized eigenvalue problems, but it is not satisfied for other important SDPs such as quadratic programming (QAP), optimal power flow and clustering with general affine constraints. 
In addition, as noted in~\cite{boumal2016global}, per iteration cost of their method for max-cut problem is an astronomical~$\mathcal{O}({d^6})$.

Lastly, there also exists a line of work for solving SDPs in their original convex formulation, in a storage efficient way~\cite{nesterov2009primal, yurtsever2015universal, yurtsever2018conditional}. These works have global optimality guarantees by their virtue of directly solving the convex formulation. On the downside, these works require the use of eigenvalue routines and exhibit significantly slower convergence as compared to nonconvex approaches~\cite{jaggi2013revisiting}.
}
}\\[-2.5em]
\section{Numerical Evidence \label{sec:experiments}}

We first begin with a caveat:  It is known that quasi-Newton methods, such as BFGS and lBFGS, might not converge for nonconvex problems~\cite{dai2002convergence, mascarenhas2004bfgs}. For this reason, we have used the trust region method as the second-order solver in our  analysis in Section~\ref{sec:cvg rate}, which is well-studied for nonconvex problems~\cite{cartis2012complexity}.  Empirically, however, BFGS and lBGFS are extremely successful and we have therefore opted for those solvers in this section since the subroutine does not affect Theorem~\ref{thm:main} as long as the subsolver performs well in practice. 

\subsection{Clustering}
Given data points $\{z_i\}_{i=1}^n $, the entries of the corresponding Euclidean distance matrix $D \in \RR^{n\times n}$ are $ D_{i, j} =  \left\| z_i - z_j\right\|^2 $. 
Clustering is then the problem of finding a co-association matrix $Y\in \RR^{n\times n}$ such that $Y_{ij} = 1$ if points $z_i$ and $z_j$ are within the same cluster and $Y_{ij} = 0$ otherwise. In~\cite{Peng2007}, the authors provide  a SDP relaxation of the clustering problem, specified as 
\begin{align}
\underset{Y \in \RR^{nxn}}{\min} \text{tr}(DY) \subjto Y\mathbf{1} = \mathbf{1},
~\text{tr}(Y) = s,~
 Y\succeq 0,~Y \geq 0,
\label{eq:sdp_svx}
\end{align}
where $s$ is the number of clusters and $Y  $ is both positive semidefinite and has nonnegative entries. 
Standard  SDP solvers do not scale well with the number of data points~$n$, since they often require projection onto the semidefinite cone with the complexity of $\mathcal{O}(n^3)$. We instead use the BM factorization to solve \eqref{eq:sdp_svx}, sacrificing convexity to reduce the computational complexity. More specifically, we solve the program
\begin{align}
\label{eq:nc_cluster}
\underset{V \in \RR^{n\times r}}{\min} \text{tr}(DVV^{\top}) \subjto
VV^{\top}\mathbf{1} = \mathbf{1},~~ \|V\|_F^2 \le  s, 
~~V  \geq 0,
\end{align}
where $\mathbf{1}\in \RR^n$ is the vector of all ones.
Note that $Y \geq 0$ in \eqref{eq:sdp_svx} is replaced above by the much stronger but easier-to-enforce constraint $V \geq 0$ in \eqref{eq:nc_cluster}, see~\cite{kulis2007fast} for the reasoning behind this relaxation.
Now, we can cast~\eqref{eq:nc_cluster} as an instance of~\eqref{prob:01}. Indeed, for every $i\le n$, let $x_i \in \RR^r$ denote the $i$th row of $V$. We next form $x \in \RR^d$ with $d = nr$ by expanding the factorized variable $V$, namely, 
$
x := [x_1^{\top}, \cdots, x_n^{\top}]^{\top} \in \RR^d,
$
and then set 
\begin{align*}
f(x) =\sum_{i,j=1}^n D_{i, j} \left\langle x_i, x_j \right\rangle,
\qquad g = \delta_C, \qquad
A(x) = [x_1^{\top}\sum_{j=1}^n x_j -1, \cdots, x_n^{\top}\sum_{j=1}^n x_j-1]^{\top},
\end{align*}
where $C$ is the intersection of the positive orthant in $\RR^d$ with the Euclidean  ball of radius $\sqrt{s}$. In Appendix~\ref{sec:verification2}, {we verify that Theorem~\ref{thm:main} applies to~\eqref{prob:01} with $f,g,A$ specified above. }

\begin{figure}[]
\begin{center}
{\includegraphics[width=.4\columnwidth]{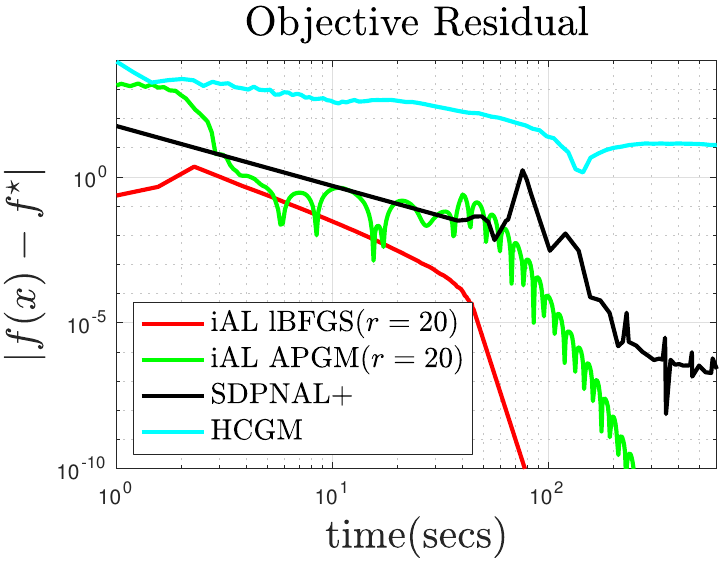}}
{\includegraphics[width=.4\columnwidth]{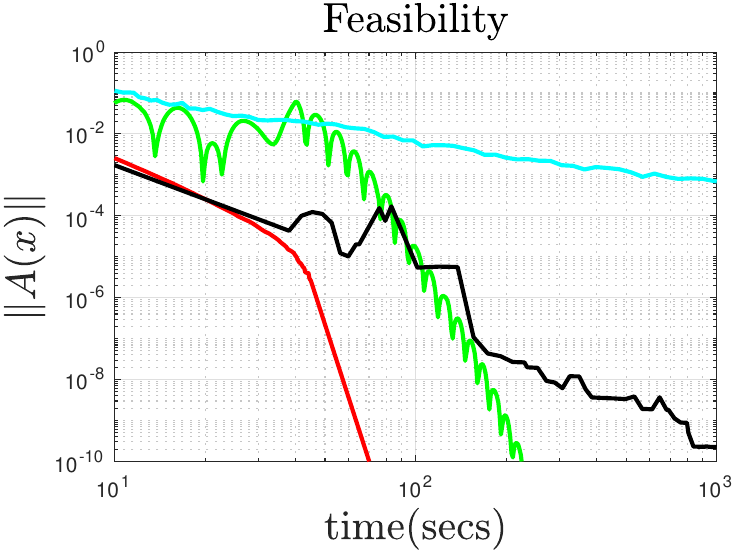}}
\caption{Clustering running time comparison. }
\label{fig:clustering}
\end{center}
\end{figure}

In our simulations, we use two different solvers for Step~2 of  Algorithm~\ref{Algo:2}, namely, APGM and lBFGS. APGM is a  solver for nonconvex problems of the form~\eqref{e:exac} with convergence guarantees to first-order stationarity, as discussed in Section~\ref{sec:cvg rate}. lBFGS is a limited-memory version of BFGS algorithm in~\cite{fletcher2013practical} that approximately leverages the second-order information of the problem. 
We compare our approach against the following convex methods:
\begin{itemize}
\item HCGM: Homotopy-based Conditional Gradient Method in~\cite{yurtsever2018conditional} which directly solves~\eqref{eq:sdp_svx}. 
\item SDPNAL+: A second-order augmented Lagrangian method for solving SDP's with nonnegativity constraints~\cite{yang2015sdpnal}. 
\end{itemize}

As for the dataset, our experimental setup is similar to that described by~\cite{mixon2016clustering}. We use the publicly-available fashion-MNIST data in \cite{xiao2017/online}, which is released as a possible replacement for the MNIST handwritten digits. Each data point is a $28\times 28$ gray-scale image, associated with a label from ten classes, labeled from $0$ to $9$. 
First, we extract the meaningful features from this dataset using a simple two-layer neural network with a sigmoid activation function. Then, we apply this neural network to 1000 test samples from the same dataset, which gives us a vector of length $10$ for each data point, where each entry represents the posterior probability for each class.  Then, we form the $\ell_2$ distance matrix ${D}$ from these probability vectors. 
The solution rank for the template~\eqref{eq:sdp_svx} is known and  it is equal to number of clusters $k$ \cite[Theorem~1]{kulis2007fast}. As discussed in~\cite{tepper2018clustering}, setting rank $r>k$ leads more accurate reconstruction in expense of speed. 
Therefore, we set the rank to 20. 
For iAL lBFGS, we used $\beta_1 = 1$ and $\sigma_1 = 10$ as the initial penalty weight and dual step size, respectively. For HCGM, we used $\beta_0 = 1$ as the initial smoothness parameter. We have run SDPNAL+ solver with $10^{-12}$ tolerance. 
The results are depicted in Figure~\ref{fig:clustering}.
We implemented 3 algorithms on MATLAB and used the software package for SDPNAL+ which contains mex files. 
It is predictable that the performance of our nonconvex approach would even improve by using mex files.

\subsection{Additional demonstrations}
We provide several additional experiments in Appendix \ref{sec:adexp}. 
Section  \ref{sec:bp} discusses a novel nonconvex relaxation of the standard basis pursuit template which performs comparable to the state of the art convex solvers. 
In Section \ref{sec:geig}, we provide fast numerical solutions to the generalized eigenvalue problem. 
In Section \ref{sec:gan}, we give a contemporary application example that our template applies, namely,  denoising with generative adversarial networks. 
Finally, we provide improved bounds for sparse quadratic assignment problem instances in Section \ref{sec:qap}.


\section{Conclusions}

In this work, we have proposed and analyzed an inexact augmented Lagrangian method for solving nonconvex optimization problems with nonlinear constraints.
We prove convergence to the first and second order stationary points of the augmented Lagrangian function, with explicit complexity estimates.
Even though the relation of stationary points and global optima is not well-understood in the literature, we find out that the algorithm has fast convergence behavior to either global minima or local minima in a wide variety of numerical experiments.

\section*{Acknowledgements}
\vspace{-2mm}
The authors would like to thank Nicolas Boumal and Nadav Hallak for the helpful suggestions. \\ 
~\\
This project has received funding from the European Research Council (ERC) under the European Union's Horizon $2020$ research and innovation programme (grant agreement n$^{\circ}$ $725594$ - time-data) and was supported by the Swiss National Science Foundation (SNSF) under  grant number $200021\_178865 / 1$. This project was also sponsored by the Department of the Navy, Office of Naval Research (ONR)  under a grant number N$62909$-$17$-$1$-$2111$ and was supported by Hasler Foundation Program: Cyber Human Systems (project number $16066$). This research was supported by the PhD fellowship program of  the Swiss Data Science Center (SDSC) under grant lD number P18-07.

\bibliographystyle{abbrv}      
\bibliography{bibliography.bib}   

\newpage
\appendix

\section{Complexity Results}\label{sec:opt_cnds}

\subsection{First-Order Optimality \label{sec:first-o-opt}}

Let us first consider the case where the solver in Step~2 is is the first-order algorithm APGM, described in detail in ~\cite{ghadimi2016accelerated}. At a high level, 
APGM makes use of $\nabla_x \mathcal{L}_{\beta}(x,y)$ in \eqref{eq:Lagrangian}, the proximal operator $\text{prox}_g$, and the classical Nesterov acceleration~\cite{nesterov1983method} to reach first-order stationarity for the subproblem in~\eqref{e:exac}.
Suppose that $g=\delta_\mathcal{X}$ is the indicator function on a bounded convex set $\mathcal{X}\subset \RR^d$ and let 
\begin{align}
\rho= \max_{x\in \mathcal{X}}\|x\|,
\end{align}
be the radius of a ball centered at the origin that includes $\mathcal{X}$. 
 Then, adapting the results in~\cite{ghadimi2016accelerated} to our setup, APGM reaches $x_{k}$ in Step 2 of Algorithm~\ref{Algo:2} {after 
\begin{equation}
\mathcal{O}\left ( \frac{\lambda_{\beta_{k}}^2 \rho^{2} }{\epsilon_{k+1}^2} \right)
\label{eq:iter_1storder}
\end{equation}}
(inner) iterations, where $\lambda_{\beta_{k}}$ denotes the Lipschitz constant of $\nabla_x{\mathcal{L}_{\beta_{k}}(x, y)}$, bounded in~\eqref{eq:smoothness of Lagrangian}. For the clarity of the presentation, we have used a looser bound in \eqref{eq:iter_1storder} compared to~\cite{ghadimi2016accelerated}. 
Using \eqref{eq:iter_1storder}, we derive the following corollary, describing the total iteration complexity of Algorithm~\ref{Algo:2} in terms of the number calls made to the first-order oracle in APGM. 

\begin{corollary}\label{cor:first_supp}
For $b>1$, let $\beta_k =b^k $ for every $k$. If we use APGM from~\cite{ghadimi2016accelerated} for Step~2 of Algorithm~\ref{Algo:2}, the algorithm finds an $(\epsilon_f,\b_k)$ first-order stationary point, 
 after $T$ calls to the first-order oracle, where
{

\begin{equation}
T = \mathcal{O}\left( \frac{Q^3 \rho^2}{\epsilon^{4}}\log_b{\left( \frac{Q}{\epsilon} \right)} \right) =   \tilde{\mathcal{O}}\left( \frac{Q^{3} \rho^2}{\epsilon^{4}} \right). 
\end{equation}}
\end{corollary}

\begin{proof}
Let $K$ denote the number of (outer) iterations of Algorithm~\ref{Algo:2} and let $\epsilon_{f}$ denote the desired accuracy of Algorithm~\ref{Algo:2}, see~(\ref{eq:inclu3}). Recalling Theorem~\ref{thm:main}, we can then write that 
\begin{equation}
 \epsilon_{f} = \frac{Q}{\b_{K}},
 \label{eq:acc_to_b}
\end{equation}
or, equivalently, $\b_{K} = Q/\epsilon_{f}$. 
We now count the number of total (inner) iterations $T$ of Algorithm~\ref{Algo:2} to reach the accuracy $\epsilon_{f}$. From \eqref{eq:smoothness of Lagrangian} and for sufficiently large $k$, recall that $\lambda_{\b_k}\le \lambda'' \b_k$ is the smoothness parameter of the augmented Lagrangian. Then, from \eqref{eq:iter_1storder} ad by summing over the outer iterations, we bound the total number of (inner) iterations of Algorithm~\ref{Algo:2}  as
{
\begin{align}\label{eq: tk_bound}
T &= \sum_{k=1}^K\mathcal{O}\left ( \frac{\lambda_{\beta_{k-1}}^2 \rho^2 }{\epsilon_k^2} \right) \nonumber\\
& = \sum_{k=1}^K\mathcal{O}\left (\beta_{k-1}^4 \rho^2  \right) 
\qquad \text{(Step 1 of Algorithm \ref{Algo:2})}
\nonumber\\
&   \leq  \mathcal{O} \left(K\beta_{K-1}^4 \rho^2 \right)
\qquad \left( \{\b_k\}_k \text{ is increasing} \right)
 \nonumber\\
 & \le \mathcal{O}\left( \frac{K Q^{{3}}  \rho^2}{\epsilon_{f}^{{4}}} \right).
 \qquad \text{(see \eqref{eq:acc_to_b})}
\end{align}
In addition, if we specify $\beta_k=b^k$ for all $k$, we can further refine $T$. Indeed, 
\begin{equation}
\beta_K = b^K~~ \Longrightarrow~~ K = \log_b \left( \frac{Q}{\epsilon_f} \right),
\end{equation}
which, after substituting into~\eqref{eq: tk_bound} gives the final bound in Corollary~\ref{cor:first}.
}
\end{proof}


\subsection{Second-Order Optimality \label{sec:second-o-opt}}
Let us now consider the second-order optimality case where the solver in Step~2 is the the trust region method developed in~\cite{cartis2012complexity}.
Trust region method minimizes a quadratic approximation of the function within a dynamically updated trust-region radius.
Second-order trust region method that we consider in this section makes use of Hessian (or an approximation of Hessian) of the augmented Lagrangian in addition to first order oracles.

As shown in~\cite{nouiehed2018convergence}, finding approximate second-order stationary points of convex-constrained problems is in general NP-hard. For this reason, we focus in this section on the special case of~\eqref{prob:01} with $g=0$. 

%

Let us compute the total computational complexity of Algorithm~\ref{Algo:2} with the trust region method in Step~2, in terms of the number of calls made to the second-order oracle.  By adapting the result in~\cite{cartis2012complexity} to our setup, we find that the number of (inner) iterations required in Step~2 of Algorithm~\ref{Algo:2} to produce $x_{k+1}$ is
\begin{equation}
\mathcal{O}\left ( \frac{\lambda_{\beta_{k}, H}^2 (\mathcal{L}_{\beta_{k}}(x_1, y) - \min_{x}\mathcal{L}_{\beta_k}(x, y))}{\epsilon_k^3} \right),
\label{eq:sec_inn_comp}
\end{equation}
%
where $\lambda_{\beta, H}$ is the Lipschitz constant of the Hessian of the augmented Lagrangian, which is of the order of $\beta$, as can be proven similar to Lemma~\ref{lem:smoothness} and $x_1$ is the initial iterate of the given outer loop.
In~\cite{cartis2012complexity}, the term $\mathcal{L}_{\beta}(x_1, y) - \min_{x}\mathcal{L}_{\beta}(x, y)$ is bounded by a constant independent of $\epsilon$.
We assume a uniform bound for this quantity for every $ \beta_k$, instead of for one value of $\beta_k$ as in~\cite{cartis2012complexity}. Using \eqref{eq:sec_inn_comp} and Theorem~\ref{thm:main}, we arrive at the following: 
%
\begin{corollary}\label{cor:second_supp}
For $b>1$, let $\beta_k =b^k $ for every $k$.
We assume that 
\begin{equation}
\mathcal{L}_{\beta}(x_1, y) - \min_{x}\mathcal{L}_{\beta}(x, y) \leq L_{u},\qquad  \forall \beta.
\end{equation}
If we use the trust region method from~\cite{cartis2012complexity} for Step~2 of Algorithm~\ref{Algo:2}, the algorithm finds an $\epsilon$-second-order stationary point of~(\ref{prob:01}) in $T$ calls to the second-order oracle where
\begin{equation}
T = \mathcal{O}\left( \frac{L_u Q'^{5}}{\epsilon^{5}} \log_b{\left( \frac{Q'}{\epsilon} \right)} \right) = \widetilde{\mathcal{O}}\left( \frac{L_u Q'^{5}}{\epsilon^{5}} \right).
\end{equation}
\end{corollary}

%
Before closing this section, we note that the remark after Corollary~\ref{cor:first} applies here as well. 

\subsection{Approximate optimality of \eqref{prob:01}.}
Corollary \ref{cor:first} establishes the iteration complexity of Algorithm~\ref{Algo:2} to reach approximate first-order stationarity for the equivalent formulation of \eqref{prob:01} presented in \eqref{eq:minmax}. Unlike the exact case, approximate first-order stationarity in  \eqref{eq:minmax} does not immediately lend itself to approximate stationarity in \eqref{prob:01}, and the study of approximate stationarity for the penalized problem (special case of our setting with dual variable set to $0$) has also precedent in~\cite{bhojanapalli2018smoothed}. 
For a precedent in convex optimization for relating the convergence in augmented Lagrangian to the constrained problem using duality, see~\cite{tran2018smooth}.
For the second-order case, it is in general not possible to establish approximate second-order optimality for \eqref{eq:minmax} from Corollary~\ref{cor:second}, with the exception of linear constraints. \cite{nouiehed2018convergence} provides an hardness result by showing that checking an approximate second-order stationarity is NP-hard.
}

\section{Proof of Theorem \ref{thm:main} \label{sec:theory}}

For every $k\ge2$, recall from (\ref{eq:Lagrangian}) and Step~2 of Algorithm~\ref{Algo:2} that $x_{k}$ satisfies 
\begin{align}
& \dist(-\nabla f(x_k) - DA(x_k)^\top y_{k-1} \nonumber\\
& \qquad - \b_{k-1} DA(x_{k})^\top A(x_k) ,\partial g(x_k) )  \nonumber\\
& = \dist(-\nabla_x \L_{\b_{k-1}} (x_k ,y_{k-1}) ,\partial g(x_k) ) \le \epsilon_{k}.  
\end{align}
With an application of the triangle inequality, it follows that 
\begin{align}
& \dist( -\b_{k-1} DA(x_k)^\top A(x_k) , \partial g(x_k) ) \nonumber\\
& \qquad \le \| \nabla f(x_k )\| + \| DA(x_k)^\top y_{k-1}\| + 
\epsilon_k,
\end{align}
which in turn implies that 
\begin{align}
& \dist( -DA(x_k)^\top A(x_k) , \partial g(x_k)/ \b_{k-1}  ) \nonumber\\
&  \le \frac{ \| \nabla f(x_k )\|}{\b_{k-1} } + \frac{\| DA(x_k)^\top y_{k-1}\|}{\b_{k-1} } + 
\frac{\epsilon_k}{\b_{k-1} } \nonumber\\
& \le \frac{\lambda'_f+\lambda'_A \|y_{k-1}\|+\epsilon_k}{\b_{k-1}} ,
\label{eq:before_restriction}
\end{align}
where $\lambda'_f,\lambda'_A$ were defined in \eqref{eq:defn_restricted_lipsichtz}. 
We next translate \eqref{eq:before_restriction} into a bound on the feasibility gap  $\|A(x_k)\|$. Using the regularity condition \eqref{eq:regularity},  the left-hand side of \eqref{eq:before_restriction} can be bounded below as 
\begin{align}
& \dist( -DA(x_k)^\top A(x_k) , \partial g(x_k)/ \b_{k-1}  ) \ge \nu \|A(x_k) \|.
\qquad \text{(see (\ref{eq:regularity}))} 
\label{eq:restrited_pre}
\end{align}
By substituting \eqref{eq:restrited_pre} back into \eqref{eq:before_restriction}, we find that 
\begin{align}
\|A(x_k)\| \le \frac{ \lambda'_f + \lambda'_A \|y_{k-1}\| + \epsilon_k}{\nu \b_{k-1} }. 
\label{eq:before_dual_controlled}
\end{align}
In words, the feasibility gap is directly controlled by the dual sequence $\{y_k\}_k$. We next establish that the dual sequence is bounded. Indeed, for every $k\in K$, note that 
\begin{align}
\|y_k\| & = \| y_0 + \sum_{i=1}^{k} \s_i A(x_i) \| 
\quad \text{(Step 5 of Algorithm \ref{Algo:2})}
\nonumber\\
& \le \|y_0\|+ \sum_{i=1}^k \s_i \|A(x_i)\|
\qquad \text{(triangle inequality)} \nonumber\\
& \le \|y_0\|+ \sum_{i=1}^k \frac{ \|A(x_1)\| \log^2 2 }{ k \log^2(k+1)} 
\quad \text{(Step 4)}
\nonumber\\
& \le \|y_0\|+ c  \|A(x_1) \| \log^2 2 =: y_{\max},
\label{eq:dual growth}
\end{align}  
where 
\begin{align}
c \ge \sum_{i=1}^{\infty} \frac{1}{k \log^2 (k+1)}. 
\end{align}
Substituting \eqref{eq:dual growth} back into \eqref{eq:before_dual_controlled}, we reach
\begin{align}
\|A(x_k)\| &  \le \frac{ \lambda'_f + \lambda'_A y_{\max} + \epsilon_k}{\nu \b_{k-1} } \nonumber\\
& \le  \frac{ 2\lambda'_f +2 \lambda'_A y_{\max} }{\nu \b_{k-1} } ,
\label{eq:cvg metric part 2}
\end{align}
where the second line above holds if $k_0$ is large enough, which would in turn guarantees that $\epsilon_k=1/\b_{k-1}$ is sufficiently small since $\{\b_k\}_k$ is increasing and unbounded. 
It remains to control the first term in \eqref{eq:cvg metric}. To that end, after recalling Step 2 of Algorithm~\ref{Algo:2} and applying the triangle inequality, we can write that 
\begin{align}
& \dist( -\nabla_x \L_{\b_{k-1}} (x_k,y_{k}),    \partial g(x_{k}) )  \nonumber\\
& \le \dist( -\nabla_x \L_{\b_{k-1}} (x_k,y_{k-1}) , \partial g(x_{k}) ) \nonumber\\
&  + \| \nabla_x \L_{\b_{k-1}} (x_k,y_{k})-\nabla_x \L_{\b_{k-1}} (x_k,y_{k-1})   \|.
\label{eq:cvg metric part 1 brk down}
\end{align}
The first term on the right-hand side above is bounded by $\epsilon_k$, by Step 5 of  Algorithm~\ref{Algo:2}. 
For the second term on the right-hand side of \eqref{eq:cvg metric part 1 brk down}, we write that 
\begin{align}
&  \| \nabla_x \L_{\b_{k-1}} (x_k,y_{k})-\nabla_x \L_{\b_{k-1}} (x_k,y_{k-1})   \|  \nonumber\\
& = \| DA(x_k)^\top (y_k - y_{k-1}) \| 
\qquad \text{(see \eqref{eq:Lagrangian})}
\nonumber\\
& \le \lambda'_A \|y_k- y_{k-1}\| 
\qquad \text{(see \eqref{eq:defn_restricted_lipsichtz})} \nonumber\\
& = \lambda'_A \s_k \|A (x_k) \|
\qquad \text{(see Step 5 of Algorithm \ref{Algo:2})} \nonumber\\
& \le \frac{2\lambda'_A \s_k }{\nu \b_{k-1} }( \lambda'_f+ \lambda'_Ay_{\max})  .
\qquad \text{(see \eqref{eq:cvg metric part 2})}
\label{eq:part_1_2}
\end{align}
By combining (\ref{eq:cvg metric part 1 brk down},\ref{eq:part_1_2}), we find that 
\begin{align}
& \dist( \nabla_x \L_{\b_{k-1}} (x_k,y_{k}),    \partial g(x_{k}) )   \nonumber\\
& \le \frac{2\lambda'_A \s_k }{\nu \b_{k-1} }( \lambda'_f+ \lambda'_Ay_{\max})  +  \epsilon_k. 
\label{eq:cvg metric part 1}
\end{align}
By combining (\ref{eq:cvg metric part 2},\ref{eq:cvg metric part 1}), we find that 
\begin{align}
& \dist( -\nabla_x \L_{\b_{k-1}}(x_k,y_k),\partial g(x_k)) + \| A(x_k)\|  \nonumber\\
& \le  \left( \frac{2\lambda'_A \s_k }{\nu \b_{k-1} }( \lambda'_f+ \lambda'_Ay_{\max})  +  \epsilon_k \right) \nonumber\\
& \qquad  + 2\left( \frac{ \lambda'_f + \lambda'_A y_{\max}}{\nu \b_{k-1} }   \right).
\end{align}
Applying  $\s_k\le \s_1$, we find that 
\begin{align}
& \dist( -\nabla_x \L_{\b_{k-1}}(x_k,y_k),\partial g(x_k)) + \| A(x_k)\|  \nonumber\\
& \le  \frac{ 2\lambda'_A\s_1 + 2}{ \nu\b_{k-1}} ( \lambda'_f+\lambda'_A y_{\max}) + \epsilon_k.
\end{align}
For the second part of the theorem, we use the Weyl's inequality and Step 5 of Algorithm~\ref{Algo:2} to write
\begin{align}\label{eq:sec}
\lambda_{\text{min}} &(\nabla_{xx} \mathcal{L}_{\beta_{k-1}}(x_k, y_{k-1})) \geq \lambda_{\text{min}} (\nabla_{xx} \mathcal{L}_{\beta_{k-1}}(x_k, y_{k})) \notag \\&-  \sigma_k \|  \sum_{i=1}^m A_i(x_k)  \nabla^2 A_i(x_k) \|.
\end{align}
The first term on the right-hand side is lower bounded by $-\epsilon_{k-1}$ by Step 2 of Algorithm~\ref{Algo:2}. We next bound the second term on the right-hand side above as 
\begin{align*}
& \sigma_k \|  \sum_{i=1}^m A_i(x_k) \nabla^2 A_i(x_k)   \|  \\
&\le \sigma_k \sqrt{m} \max_{i} \| A_i(x_k)\| \| \nabla^2 A_i(x_k)\|   \\ 
&\le \sigma_k \sqrt{m} \lambda_A \frac{ 2\lambda'_f +2 \lambda'_A y_{\max} }{\nu \b_{k-1} },
\end{align*}
where the last inequality is due to~(\ref{eq:smoothness basic},\ref{eq:cvg metric part 2}). 
Plugging into~\eqref{eq:sec} gives
\begin{align*}
& \lambda_{\text{min}}(\nabla_{xx} \mathcal{L}_{\beta_{k-1}}(x_k, y_{k-1}))\nonumber\\
&  \geq -\epsilon_{k-1} - \sigma_k \sqrt{m} \lambda_A \frac{ 2\lambda'_f +2 \lambda'_A y_{\max} }{\nu \b_{k-1} },
\end{align*}
which completes the proof of Theorem \ref{thm:main}.

\section{Proof of Lemma \ref{lem:smoothness}\label{sec:proof of smoothness lemma}}
\begin{proof}
Note that
\begin{align}
\mathcal{L}_{\beta}(x,y) = f(x) + \sum_{i=1}^m y_i A_i (x)  + \frac{\b}{2} \sum_{i=1}^m (A_i(x))^2,
\end{align}
which implies that 
\begin{align}
& \nabla_x \mathcal{L}_\beta(x,y) \nonumber\\
& = \nabla f(x) + \sum_{i=1}^m y_i  \nabla A_i(x) + \frac{\b}{2} \sum_{i=1}^m A_i(x) \nabla A_i(x) \nonumber\\
& = \nabla f(x) + DA(x)^\top y  + \b DA(x)^\top A(x),
\end{align}
where $DA(x)$ is the Jacobian of $A$ at $x$. By taking another derivative with respect to $x$, we reach 
\begin{align}
\nabla^2_x \mathcal{L}_\beta(x,y) & = \nabla^2 f(x) + \sum_{i=1}^m \left( y_i + \b A_i(x) \right) \nabla^2 A_i(x) \nonumber\\ 
& \qquad +\b \sum_{i=1}^m \nabla A_i(x) \nabla A_i(x)^\top. 
\end{align}
It follows that 
\begin{align}
& \|\nabla_x^2 \mathcal{L}_\beta(x,y)\|\nonumber\\
 & \le \| \nabla^2 f(x) \| +  \max_i \| \nabla^2 A_i(x)\| \left (\|y\|_1+\b \|A(x)\|_1 \right) \nonumber\\
& \qquad +\beta\sum_{i=1}^m \|\nabla A_i(x)\|^2 \nonumber\\
& \le \lambda_h+  \sqrt{m} \lambda_A \left (\|y\|+\b \|A(x)\| \right) + \b \|DA(x)\|^2_F.
\end{align}
For every $x$ such that $\|x\|\le \rho$ and $\|A(x)\|\le \rho$, we conclude that 
\begin{align}
\|\nabla_x^2 \mathcal{L}_\beta(x,y)\|
& \le \lambda_f + \sqrt{m}\lambda_A \left(\|y\| + \b\rho' \right)+ \b \max_{\|x\|\le \rho}\|DA(x)\|_F^2,
\end{align}
which completes the proof of Lemma \ref{lem:smoothness}. 
\end{proof}

\section{Clustering \label{sec:verification2}}

We only verify the condition in~\eqref{eq:regularity} here. 
Note that
\begin{align}
A(x) = VV^\top \mathbf{1}-  \mathbf{1},
\end{align}
\begin{align}
DA(x) & = 
\left[
\begin{array}{cccc}
w_{1,1} x_1^\top & \cdots & w_{1,n} x_{1}^\top\\
\vdots\\
w_{n,1}x_{n}^\top &  \cdots & w_{n,n}1 x_{n}^\top
\end{array}
\right] \nonumber\\
& = \left[ 
\begin{array}{ccc}
V & \cdots & V
\end{array}
\right]
+ 
\left[
\begin{array}{ccc}
x_1^\top & \\
& \ddots & \\
& & x_n^\top
\end{array}
\right],
\label{eq:Jacobian clustering}
\end{align}
where $w_{i.i}=2$ and $w_{i,j}=1$ for $i\ne j$. In the last line above, $n$ copies of $V$ appear and the last matrix above is block-diagonal. For $x_k$, define $V_k$ accordingly and  let $x_{k,i}$ be the $i$th row of $V_k$. 
Consequently,
\begin{align}
DA(x_k)^\top A(x_k) & = 
\left[
\begin{array}{c}
(V_k^\top V_k - I_n) V_k^\top \mathbf{1}\\
\vdots\\
(V_k^\top V_k - I_n) V_k^\top \mathbf{1}
\end{array}
\right] \nonumber\\
& \qquad + 
\left[
\begin{array}{c}
x_{k,1} (V_k V_k^\top \mathbf{1}-   \mathbf{1})_1 \\  
\vdots \\
x_{k,n} (V_k V_k^\top \mathbf{1}-   \mathbf{1})_n 
\end{array}
\right],
\end{align}
where $I_n\in \RR^{n\times n}$ is the identity matrix. 
Let us make a number of simplifying assumptions. First, we assume that $\|x_k\|< \sqrt{s}$ (which can be  enforced in the iterates by replacing $C$ with $(1-\epsilon)C$ for a small positive $\epsilon$ in the subproblems).  Under this assumption, it follows that 
\begin{align}
(\partial g(x_k))_i =
\begin{cases}
0 & (x_k)_i > 0\\
\{a: a\le 0\} & (x_k)_i = 0,
\end{cases}
\qquad i\le d.
\label{eq:exp-subgrad-cluster}
\end{align}
Second, we assume that $V_k$ has nearly orthonormal columns, namely, $V_k^\top V_k \approx I_n$. This can also be  enforced in each iterate of Algorithm~\ref{Algo:2} and naturally corresponds to well-separated clusters. While a more fine-tuned argument can remove these assumptions, they will help us simplify the presentation here. Under these assumptions, the (squared) right-hand side of \eqref{eq:regularity} becomes
\begin{align}
& \dist\left( -DA(x_k)^\top A(x_k) , \frac{\partial g(x_k)}{ \b_{k-1}}  \right)^2 \nonumber\\
& = \left\| \left( -DA(x_k)^\top A(x_k) \right)_+\right\|^2
\qquad (a_+ = \max(a,0))
 \nonumber\\
& = 
\left\| 
\left[
\begin{array}{c}
x_{k,1}  (V_k V_k^\top \mathbf{1}-   \mathbf{1})_1 \\  
\vdots \\
x_{k,n}  (V_k V_k^\top \mathbf{1}-   \mathbf{1})_n 
\end{array}
\right]
\right\|^2 
\qquad (x_k\in C \Rightarrow x_k\ge 0)
\nonumber\\
& = \sum_{i=1}^n \| x_{k,i}\|^2 (V_kV_k^\top \mathbf{1}-\mathbf{1})_i^2 \nonumber\\
& \ge \min_i \| x_{k,i}\|^2
\cdot \sum_{i=1}^n  (V_kV_k^\top \mathbf{1}-\mathbf{1})_i^2 \nonumber\\
& = \min_i \| x_{k,i}\|^2
\cdot \| V_kV_k^\top \mathbf{1}-\mathbf{1} \|^2.
\label{eq:final-cnd-cluster}
\end{align}
Therefore, given a prescribed $\nu$, ensuring $\min_i \|x_{k,i}\| \ge \nu$ guarantees \eqref{eq:regularity}. When the algorithm is initialized close enough to the constraint set, there is indeed no need to separately enforce \eqref{eq:final-cnd-cluster}. In practice, often $n$ exceeds the number of true clusters and a more intricate analysis is required to establish \eqref{eq:regularity} by restricting the argument to a particular subspace of $\RR^n$.

\section{Additional Experiments}{\label{sec:adexp}}

\subsection{Basis Pursuit}{\label{sec:bp}}

Basis Pursuit (BP) finds sparsest solutions of an under-determined system of linear equations by solving
\begin{align}
\min_{z} \|z\|_1 \subjto
Bz = b,
\label{eq:bp_main}
\end{align}
where $B \in \RR^{n \times d}$ and $b \in \RR^{n}$. 
Various primal-dual convex optimization algorithms are available in the literature to solve BP, including~\cite{tran2018adaptive,chambolle2011first}.  
We compare our algorithm against state-of-the-art primal-dual convex methods for solving \eqref{eq:bp_main}, namely, Chambole-Pock~\cite{chambolle2011first}, ASGARD~\cite{tran2018smooth} and ASGARD-DL~\cite{tran2018adaptive}. 


Here, we take a different approach and cast~(\ref{eq:bp_main}) as an instance of~\eqref{prob:01}. Note that any $z \in \RR^d$ can be decomposed as $z = z^+ - z^-$, where $z^+,z^-\in \RR^d$ are the positive and negative parts of $z$, respectively. Then consider the change of variables $z^+ = u_1^{\circ 2}$ and  $z^-= u_2^{\circ 2} \in \RR^d$, where $\circ$ denotes element-wise power. Next, we concatenate $u_1$ and $u_2$ as $x := [ u_1^{\top}, u_2^{\top} ]^{\top} \in \RR^{2d}$ and define $\overline{B} := [B, -B] \in \RR^{n \times 2d}$. Then, \eqref{eq:bp_main} is equivalent to \eqref{prob:01}  with 
\begin{align}
f(x) =& \|x\|^2, \quad g(x) = 0,\subjto
A(x) = \overline{B}x^{\circ 2}- b.
\label{eq:bp-equiv}
\end{align}

We draw the entries of $B$ independently from a zero-mean and unit-variance Gaussian distribution.  For a fixed sparsity level $k$, the support of $z_*\in \RR^d$ and its nonzero amplitudes are also drawn from the standard Gaussian distribution. 
Then the measurement vector  is created as $b = Bz + \epsilon$, where $\epsilon$ is the noise vector with entries drawn independently from the zero-mean Gaussian distribution with variance $\sigma^2=10^{-6}$. 

The results are compiled in Figure~\ref{fig:bp1}. Clearly,  the performance of Algorithm~\ref{Algo:2} with a second-order solver for BP is comparable to the rest. 
It is, indeed, interesting to see that these type of nonconvex relaxations gives the solution of convex one and first order methods succeed.

\begin{figure}[]
\centering
{\includegraphics[width=.4\columnwidth]{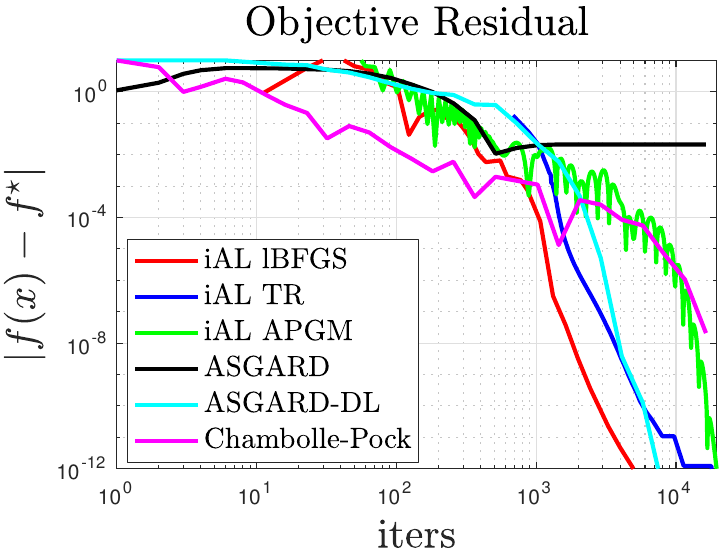}}
{\includegraphics[width=.4\columnwidth]{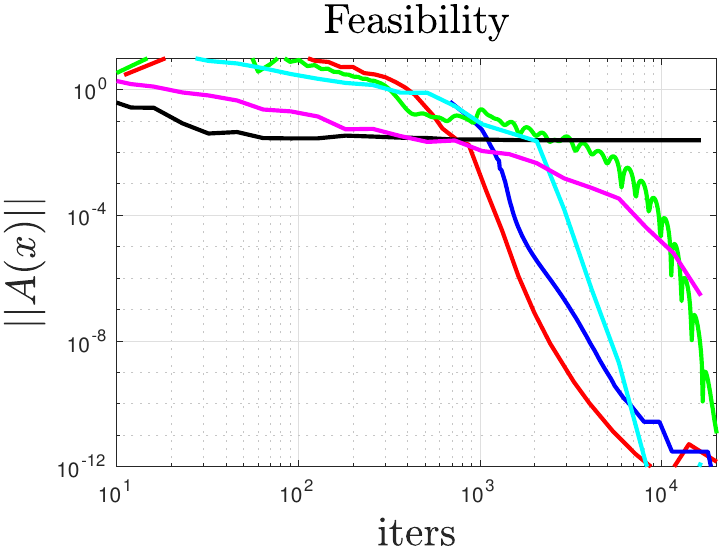}}
\caption{Basis Pursuit}
\label{fig:bp1}
\end{figure}

\paragraph{Discussion:} 
The true potential of our reformulation is in dealing with more structured norms rather than $\ell_1$, where computing the proximal operator is often intractable. One such case is the latent group lasso norm~\cite{obozinski2011group}, defined as 
\begin{align*}
\|z\|_{\Omega} = \sum_{i=1}^I \| z_{\Omega_i}  \|, 
\end{align*}
{where $\{\Omega_i\}_{i=1}^I$ are (not necessarily disjoint) index sets of $\{1,\cdots,d\}$. Although not studied here, we believe that the nonconvex framework presented in this paper can serve to solve more complicated problems, such as the latent group lasso. We leave this research direction for future work.

\paragraph{Condition verification:}
In the sequel, we verify that Theorem~\ref{thm:main} indeed applies to~\eqref{prob:01} with the above $f,A,g$. 
Note that 
\begin{align}
DA(x) = 2 \overline{B} \text{diag}(x), 
\label{eq:jacob-bp}
\end{align}
where $\text{diag}(x)\in\RR^{2d\times 2d}$ is the diagonal matrix formed by $x$. The left-hand side of \eqref{eq:regularity} then reads as 
\begin{align}
& \text{dist} \left(  -DA(x_k)^\top A(x_k) , \frac{\partial g(x_k)}{\b_{k-1}}  \right) \nonumber\\
& = \text{dist} \left(  -DA(x_k)^\top A(x_k) , \{0\}  \right) \qquad (g\equiv 0)\nonumber\\
& = \|DA(x_k)^\top A(x_k)   \| \nonumber\\
& =2  \|  \text{diag}(x_k) \overline{B}^\top ( \overline{B}x_k^{\circ 2} -b) \|.
\qquad \text{(see \eqref{eq:jacob-bp})}
\label{eq:cnd-bp-pre}
\end{align}
To bound the last line above, let $x_*$ be a solution of~\eqref{prob:01} and note that $\overline{B} x_*^{\circ 2} = b $ by definition.  Let also $z_k,z_*\in \RR^d$ denote the vectors corresponding to $x_k,x_*$. Corresponding to $x_k$, also define $u_{k,1},u_{k,2}$ naturally and let $|z_k| = u_{k,1}^{\circ 2} + u_{k,2}^{\circ 2} \in \RR^d$ be the vector of amplitudes of $z_k$. To simplify matters, let us assume also that $B$ is full-rank. 
We then rewrite the norm in the last line of \eqref{eq:cnd-bp-pre} as 
\begin{align}
&   \|  \text{diag}(x_k) \overline{B}^\top ( \overline{B}x_k^{\circ 2} -b) \|^2 \nonumber\\
& =   \|  \text{diag}(x_k) \overline{B}^\top  \overline{B} (x_k^{\circ 2} -x_*^{\circ 2}) \|^2 
\qquad (\overline{B} x_*^{\circ 2} = b) 
\nonumber\\
& =  \|  \text{diag}(x_k)\overline{B}^\top B (x_k - x_*) \|^2 \nonumber\\
& =  \|  \text{diag}(u_{k,1})B^\top B (z_k - z_*) \|^2 \nonumber\\
& \qquad + \|  \text{diag}(u_{k,2})B^\top B (z_k - z_*) \|^2 \nonumber\\
& =  \|  \text{diag}(u_{k,1}^{\circ 2}+ u_{k,2}^{\circ 2}) B^\top B (z_k - z_*) \|^2 \nonumber\\
& =  \|  \text{diag}(|z_k|) B^\top B (z_k - z_*) \|^2 \nonumber\\
& \ge 
\eta_n (   B \text{diag}(|z_k|) )^2
\| B(z_k - z_*) \|^2 \nonumber\\
& =
\eta_n (   B \text{diag}(|z_k|) )^2
\| B z_k -b \|^2
\qquad ( Bz_* = \ol{B} x^{\circ2}_* = b) \nonumber\\
& \ge \min_{|T|=n}\eta_n(B_T)\cdot  |z_{k,(n)}|^2 \|Bz_k - b\|^2,
\end{align}
where $\eta_n(\cdot)$ returns the $n$th largest singular value of its argument. In the last line above, $B_T$ is the restriction of $B$ to the columns indexed by $T$ of size $n$. Moreover, $z_{k,(n)}$ is the $n$th largest entry of $z$ in magnitude. 
Given  a prescribed $\nu$,  \eqref{eq:regularity} therefore holds if 
\begin{align}
|z_{k,(n)} | \ge \frac{\nu}{2 \sqrt{\min_{|T|=n} \eta_n(B_T)}} ,
\label{eq:final-bp-cnd}
\end{align}
for every iteration $k$. If Algorithm \ref{Algo:2} is initialized close enough to the solution $z^*$ and the entries of $z^*$ are sufficiently large in magnitude, there will be no need to directly enforce \eqref{eq:final-bp-cnd}.

\subsection{Generalized Eigenvalue Problem}{\label{sec:geig}}

\begin{figure}[!h]
\begin{tabular}{l|l|l}
~~~~~~~~~(i) $C:$ Gaussian iid & ~~~~~(ii) $C:$ Polynomial decay & ~~~~~(iii) $C:$ Exponential decay  \\

\includegraphics[width=.3\columnwidth]{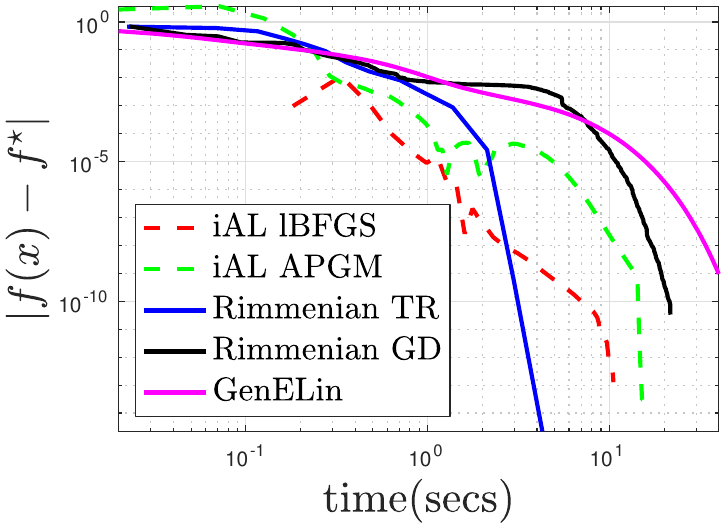}&
\includegraphics[width=.3\columnwidth]{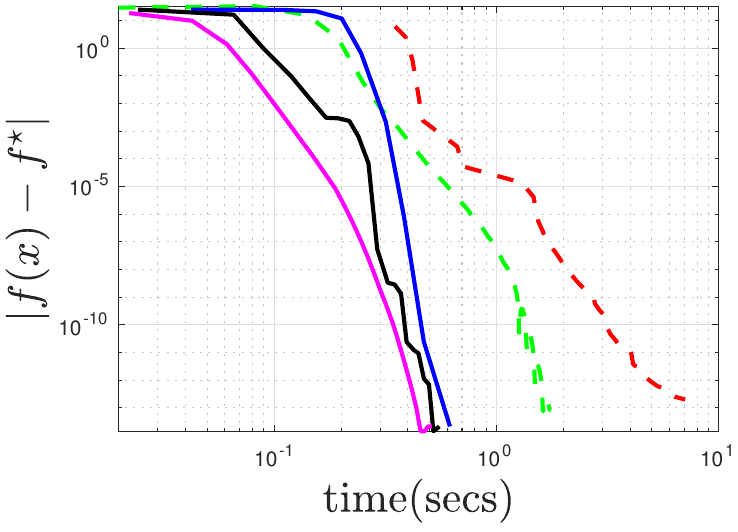}& 
\includegraphics[width=.3\columnwidth]{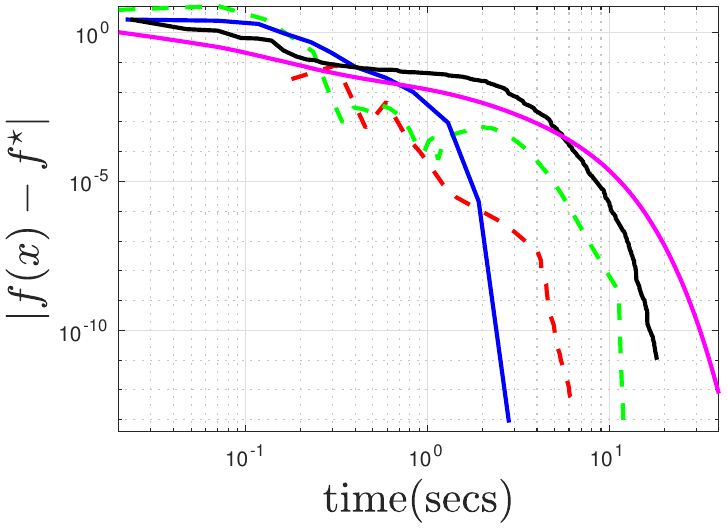}\\

\includegraphics[width=.31\columnwidth]{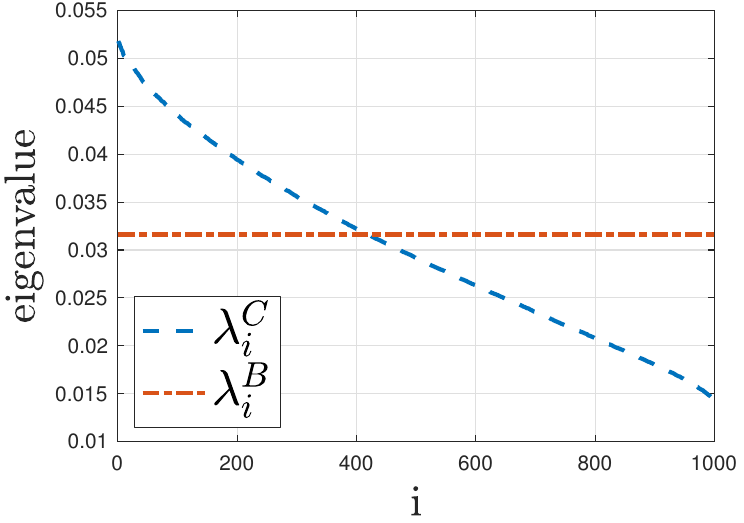} &
\includegraphics[width=.31\columnwidth]{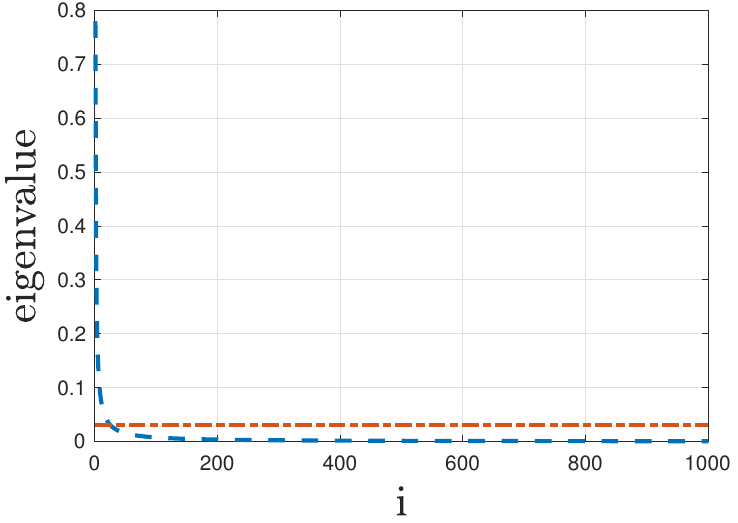}  & 
\includegraphics[width=.31\columnwidth]{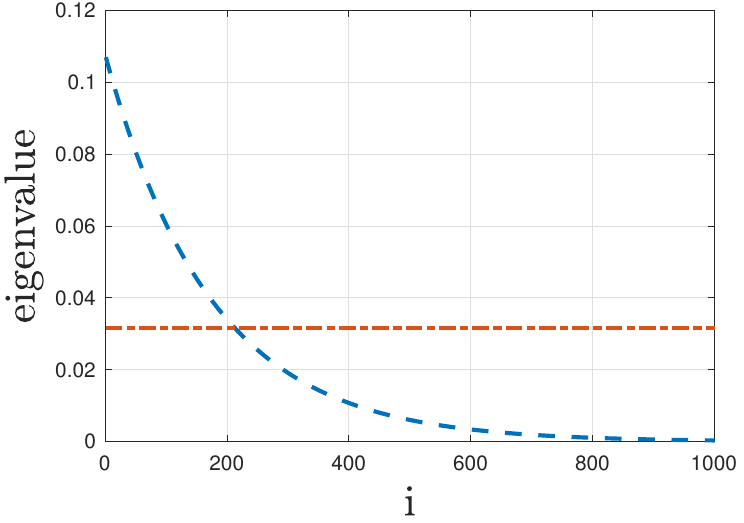}  \\ \hline
~~~~~~~~~~~~~~~~~~~~(iv) & ~~~~~~~~~~~~~~~~~~~~(v) & ~~~~~~~~~~~~~~~~~~~(vi) \\

\includegraphics[width=.3\columnwidth]{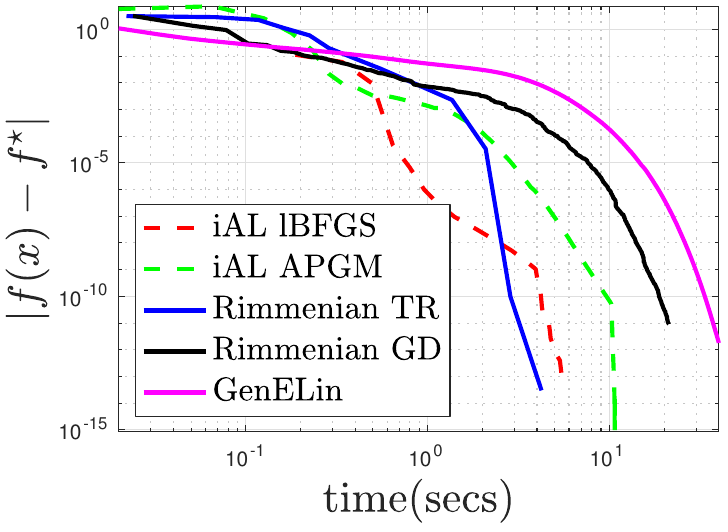}&
\includegraphics[width=.3\columnwidth]{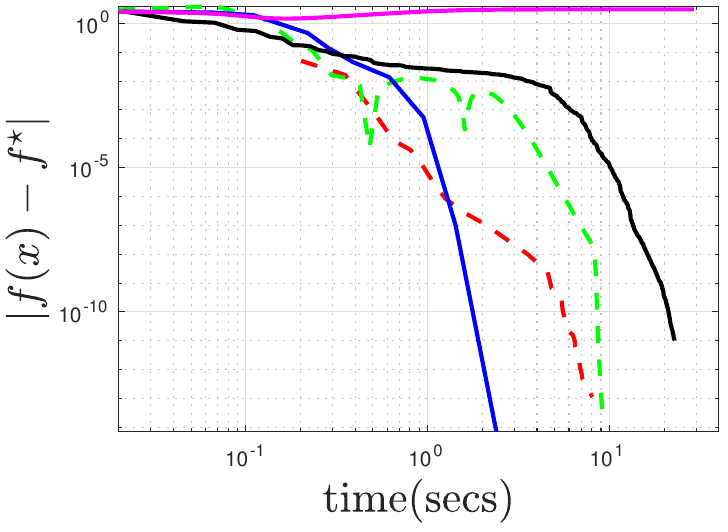}& 
\includegraphics[width=.3\columnwidth]{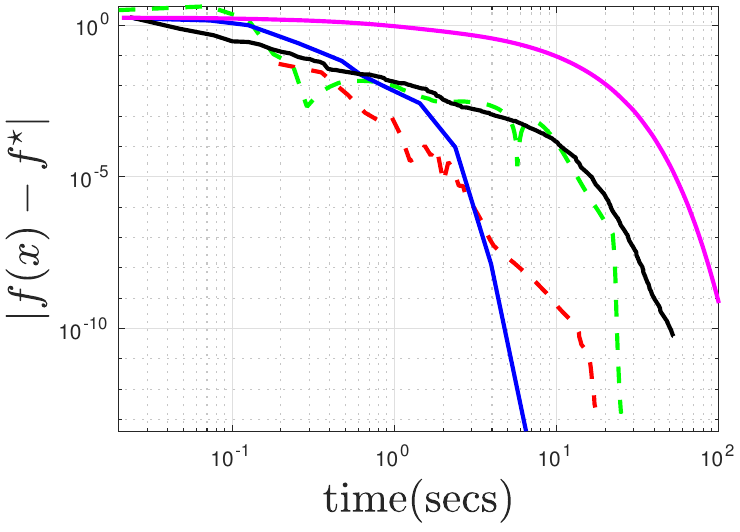}\\

\includegraphics[width=.31\columnwidth]{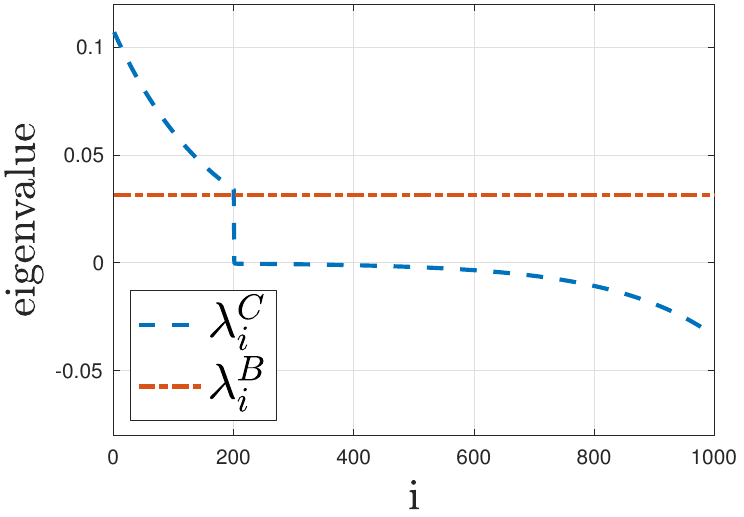} &
\includegraphics[width=.31\columnwidth]{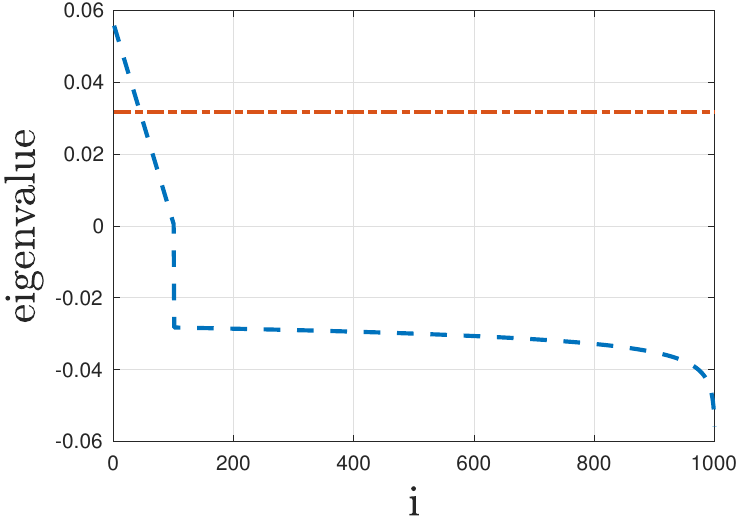}  & 
\includegraphics[width=.31\columnwidth]{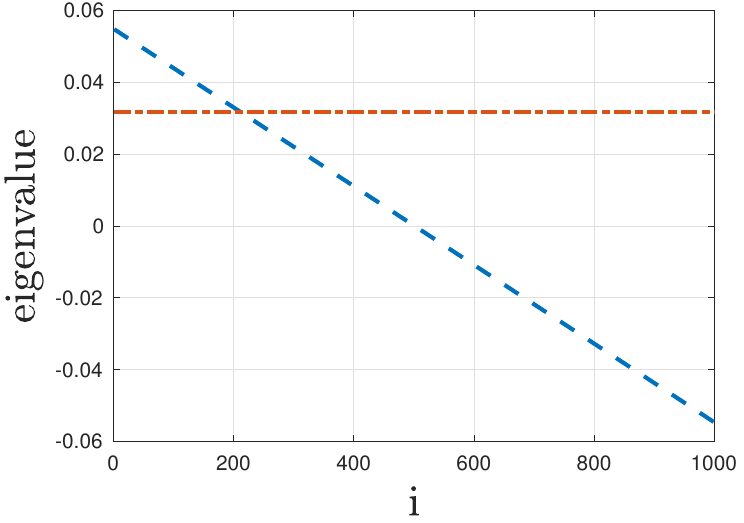} 
 
\end{tabular}
\caption{{\it{(Top)}} Objective convergence for calculating top generalized eigenvalue and eigenvector of $B$ and $C$. {\it{(Bottom)}} Eigenvalue structure of the matrices. For (i),(ii) and (iii), $C$ is positive semidefinite; for (iv), (v) and (vi), $C$ contains negative eigenvalues. {[(i): Generated by taking symmetric part of iid Gaussian matrix. (ii): Generated by randomly rotating diag($1^{-p}, 2^{-p}, \cdots, 1000^{-p}$)($p=1$). (iii): Generated by randomly rotating diag($10^{-p}, 10^{-2p}, \cdots, 10^{-1000p}$)($p=0.0025$).]} 
}
\label{fig:geig1}
\end{figure}

Generalized eigenvalue problem has extensive applications in machine learning, statistics and data analysis~\cite{ge2016efficient}.
The well-known nonconvex formulation of the problem is~\cite{boumal2016non} given by
\begin{align}
\begin{cases}
\underset{x\in\mathbb{R}^n}{\min} x^\top C x \\ 
x^\top B x = 1,
\end{cases}
\label{eq:eig}
\end{align}
where $B, C \in \RR^{n \times n}$ are symmetric matrices and $B$ is positive definite, namely, $B \succ 0$.
The generalized eigenvector computation is equivalent to performing principal component analysis (PCA) of $C$ in the norm $B$. It is also equivalent to computing the top eigenvector of symmetric matrix $S = B^{-1/2}CB^{1/2}$ and multiplying the resulting vector by $B^{-1/2}$. However, for  large values of $n$, computing $B^{-1/2}$ is extremely expensive. 
The natural convex SDP relaxation for~\eqref{eq:eig} involves lifting $Y = xx^\top$ and removing the nonconvex rank$(Y) = 1$ constraint, namely, 
\begin{align}
\begin{cases}
\underset{Y \in \RR^{n \times n}}{\min} \text{tr}(CY)\\
\text{tr}(BY) = 1, \quad X \succeq 0.
\end{cases}
\label{eq:eig-sdp}
\end{align}
Here, however, we opt to directly solve~\eqref{eq:eig} because it  fits into our template with
\begin{align}
f(x) =& x^\top C x, \quad g(x) = 0,\nonumber\\
A(x) =& x^\top B x - 1.
\label{eq:eig-equiv}
\end{align}
We compare our approach against three different methods: manifold based Riemannian gradient descent and Riemannian trust region methods in \cite{boumal2016global} and  the linear system solver in~\cite{ge2016efficient}, abbrevated as GenELin.  We have used Manopt software package in \cite{boumal2014manopt} for the manifold based methods. For GenELin, we have utilized Matlab's backslash operator as the linear solver. The results are compiled in Figure \ref{fig:geig1}.

\paragraph{Condition verification:}
Here, we verify the regularity condition in \eqref{eq:regularity} for problem \eqref{eq:eig}. Note that

\begin{align}
DA(x) = (2Bx)^\top.
\label{eq:jacobian-gen-eval}
\end{align}
Therefore,
\begin{align}
 \dist\left( -DA(x_k)^\top A(x_k) , \frac{\partial g(x_k)}{ \b_{k-1}}  \right)^2 & = \dist\left( -DA(x_k)^\top A(x_k) , \{0\}  \right)^2 
\qquad (g\equiv 0) 
 \nonumber\\
& = \| DA(x_k)^\top A(x_k) \|^2 
\nonumber\\
& = \|2Bx_k (x_k^\top Bx_k - 1)\|^2 
\qquad \text{(see \eqref{eq:jacobian-gen-eval})}
\nonumber\\
& = 4 (x_k^\top Bx_k - 1)^2\|Bx_k\|^2 
\nonumber\\
& = 4\|Bx_k\|^2 \|A(x_k)\|^2 
\qquad \text{(see \eqref{eq:eig-equiv})}
\nonumber \\
& \ge \eta_{\min}(B)^2\|x_k\|^2 \|A(x_k)\|^2,
\end{align}
where $\eta_{\min}(B)$ is the smallest eigenvalue of the positive definite matrix $B$. 
Therefore, for a prescribed $\nu$, the regularity condition in \eqref{eq:regularity} holds with $\|x_k\| \ge \nu /\eta_{min}$ for every $k$. If the algorithm is initialized close enough to the constraint set, there will be again no need to directly enforce this latter condition.

\subsection{$\ell_\infty$ Denoising with a Generative Prior}{\label{sec:gan}}
The authors of \cite{Samangouei2018,Ilyas2017} have proposed to
project onto the range of a Generative Adversarial network (GAN)
\cite{Goodfellow2014}, as a way to defend against adversarial examples. For a
given noisy observation $x^* + \eta$, they consider a projection in the
$\ell_2$ norm. We instead propose to use our augmented Lagrangian method to
denoise in the $\ell_\infty$ norm, a much harder task:
\begin{align}
\begin{array}{lll}
\underset{x, z}{\text{min}} & & \|x^* + \eta - x\|_\infty \\
\text{s.t. } && x=G(z).
\end{array}
\end{align}

\begin{figure}[!h]
\begin{center}
{\includegraphics[scale=0.9]{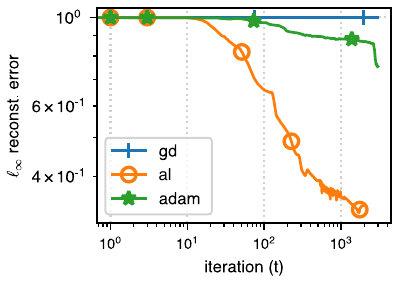}}
\caption{Augmented Lagrangian vs Adam and Gradient descent for $\ell_\infty$ denoising}
\label{fig:comparison_fab}
\end{center}
\end{figure}

We use a pretrained generator for the MNIST dataset, given by a standard
deconvolutional neural network architecture \cite{Radford2015}. We compare the
succesful optimizer Adam \cite{Kingma2014} and gradient Descent against our
method. Our algorithm involves two forward and one backward pass through the
network, as oposed to Adam that requires only one forward/backward pass. For
this reason we let our algorithm run for 2000 iterations, and Adam and GD for
3000 iterations. Both Adam and gradient descent generate a sequence of feasible
iterates $x_t=G(z_t)$. For this reason we plot the objective evaluated at the
point $G(z_t)$ vs iteration count in figure \ref{fig:comparison_fab}. Our
method successfully minimizes the objective value, while Adam and GD do not.

\subsection{Quadratic assginment problem}{\label{sec:qap}}

Let $K$, $L$ be $n \times n$ symmetric metrices. QAP in its simplest form can be written as

\begin{align}
\max \text{tr}(KPLP), \,\,\,\,\,\,\, \text{subject to}\,\, P \,\,\text{be a permutation matrix}
\label{eq:qap1}
\end{align}

A direct approach for solving \eqref{eq:qap1} involves a combinatorial search. 
To get the SDP relaxation of \eqref{eq:qap1}, we will first lift the QAP to a problem involving a larger
matrix. Observe that the objective function takes the form

\begin{align*}
\text{tr}((K\otimes L) (\text{vec}(P)\text{vec}(P^\top))),
\end{align*}
 where $\otimes$ denotes the Kronecker product. Therefore, we can recast \eqref{eq:qap1} as 
 
 \begin{align}
\text{tr}((K\otimes L) Y) \,\,\,\,\text{subject to} \,\,\, Y = \text{vec}(P)\text{vec}(P^\top),
\label{eq:qapkron}
\end{align}
where $P$ is a permutation matrix. We can relax the equality constraint in \eqref{eq:qapkron} to a semidefinite constraint and write it in an equivalent form as

 \begin{align*}
X = 
\begin{bmatrix} 
1 & \text{vec}(P)^\top\\
 \text{vec}(P) & Y 
\end{bmatrix}
\succeq 0 \,\,\, \text{for a symmetric} X \in \mathbb{S}^{(n^2+1) \times (n^2+1)} 
\end{align*}

We now introduce the following constraints such that 
\begin{equation}
B_k(X) = {\bf{b_k}},  \,\,\,\, {\bf b_k} \in \RR^{m_k}\,\, 
\label{eq:qapcons}
\end{equation}
to make sure X has a proper structure. Here, $B_k$ is a linear operator on $X$ and the total number of constraints is $m = \sum_{k} m_k.$ Hence, SDP relaxation of the quadratic assignment problem takes the form,  


\begin{align}
\max\,\,\, & \langle C, X \rangle \nonumber \\
 \text{subject to} \,\,\,& P1 = 1, \,\, 1^\top P = 1,\,\, P\geq 0 \nonumber \\
 & \text{trace}_1(Y) = I \,\,\, \text{trace}_2(Y) = I \nonumber \\
 & \text{vec}(P) = \text{diag}(Y) \nonumber \\
& \text{trace}(Y) = n \,\, \begin{bmatrix} 
1 & \text{vec}(P)^\top\\
 \text{vec}(P) & Y 
\end{bmatrix}
\succeq 0,  
\label{eq:qap_sdp}
\end{align}
where $\text{trace}_1(.)$ and $\text{trace}_2(.)$ are partial traces satisfying,
$$
\text{trace}_1(K \otimes L) = \text{trace}(K)L \,\,\,\,\,\, \text{and} \,\,\,\,\, \text{trace}_2(K\otimes L)=  K \text{trace}(L)
$$
$$
\text{trace}_1^*(T) = I \otimes T  \,\,\,\,\,\, \text{and} \,\,\,\,\, \text{trace}_2^*(T)=  T \otimes I
$$
$1st$ set of equalities are due to the fact that permutation matrices are doubly stochastic. $2nd$ set of equalities are to ensure permutation matrices are orthogonal, i.e., $PP^\top = P^\top P=I$. $3rd$ set of equalities are to enforce every individual entry of the permutation matrix takes either $0$ or $1$,  i.e., $X_{1, i} = X_{i,i} \,\,\, \forall i \in [1, n^2+1]$. . Trace constraint in the last line is to bound the problem domain. 
By concatenating the $B_k$'s in \eqref{eq:qapcons}, we can rewrite \eqref{eq:qap_sdp} in standard SDP form as 

\begin{align}
\max\,\,\, & \langle C, X \rangle \nonumber \\
 \text{subject to} \,\,\,& B(X) = {\bf{b},\,\,\, b} \in \RR^m  \nonumber \\
 & \text{trace}(X) = n+1 \nonumber \\
& X_{ij} \geq 0, \,\,\,\, i,j\,\,\, \mathcal{G}\nonumber  \\
 & X\succeq0,
\end{align}
where $\mathcal{G}$ represents the index set for which we introduce the nonnegativities. When $\mathcal{G}$ covers the wholes set of indices, we get the best approximation to the original problem. However, it becomes computationally undesirable as the problem dimension increases. Hence, we remove the redundant nonnegativity constraints and enforce it for the indices where Kronecker product between $K$ and $L$ is nonzero. 

We penalize the non-negativity constraints and add it to the  augmented Lagrangian objective since a projection to the positive orthant approach in the low rank space as we did for the clustering does not work here. 

We take \cite{ferreira2018semidefinite} as the baseline. This is an SDP based approach for solving QAP problems containing a sparse graph. 
We compare against the best feasible upper bounds reported in \cite{ferreira2018semidefinite} for the given instances. Here, optimality gap is defined as 
$$
\% \text{Gap} = \frac{|\text{bound} - \text{optimal}|}{\text{optimal}} \times 100
$$
We used a (relatively) sparse graph data set from the QAP library.
We run our low rank algorithm for different rank values. $r_m$ in each instance corresponds to the smallest integer satisfying the Pataki bound~\cite{pataki1998rank, barvinok1995problems}. 
Results are shown in Table \ref{tb:qap}. 
Primal feasibility values except for the last instance $esc128$ is less than $10^{-5}$ and we obtained bounds at least as good as the ones reported in \cite{ferreira2018semidefinite} for these problems.

For $esc128$, the primal feasibility is $\approx 10^{-1}$, hence, we could not manage to obtain a good optimality gap.

\begin{table}[]
\begin{tabular}{|l|l|l|l|l|l|l|l|}
\hline
& \multicolumn{1}{l|}{}              & \multicolumn{6}{c|}{Optimality Gap ($\%$)}                                                                                                                                                   \\ \hline
\multicolumn{1}{|c|}{\multirow{2}{*}{Data}} & \multicolumn{1}{c|}{\multirow{2}{*}{Optimal Value}} & \multicolumn{1}{c|}{\multirow{2}{*}{Sparse QAP~\cite{ferreira2018semidefinite}}} & \multicolumn{5}{c|}{iAL}    \\ \cline{4-8} 
\multicolumn{1}{|c|}{}                      & \multicolumn{1}{c|}{}                               & \multicolumn{1}{c|}{}                            & $r=10$ & $r=25$ & $r=50$ & $r=r_m$ & $r_m$ \\ \hline
esc16a                                      & 68                                                  & 8.8                                              & 11.8   & $\mathbf{0}$    & $\mathbf{0}$    & 5.9     & 157   \\ \hline
esc16b                                      & 292                                                 & $\mathbf{0}$                              & $\mathbf{0}$    & $\mathbf{0}$   & $\mathbf{0}$    & $\mathbf{0}$       & 224   \\ \hline
esc16c                                      & 160                                                 & 5                                                & 5.0    & 5.0    & $\mathbf{2.5}$    & 3.8     & 177   \\ \hline
esc16d                                      & 16                                                  & 12.5                                             & 37.5   & $\mathbf{0}$   & $\mathbf{0}$    & 25.0    & 126   \\ \hline
esc16e                                      & 28                                                  & 7.1                                              & 7.1    &  $\mathbf{0}$    & 14.3   & 7.1     & 126   \\ \hline
esc16g                                      & 26                                                  & $\mathbf{0}$                         & 23.1   & 7.7    &  $\mathbf{0}$    &  $\mathbf{0}$     & 126   \\ \hline
esc16h                                      & 996                                                 &  $\mathbf{0}$                          &  $\mathbf{0}$   &  $\mathbf{0}$   &  $\mathbf{0}$    &  $\mathbf{0}$     & 224   \\ \hline
esc16i                                      & 14                                                  &  $\mathbf{0}$                 &  $\mathbf{0}$    &  $\mathbf{0}$    & 14.3   &  $\mathbf{0}$    & 113   \\ \hline
esc16j                                      & 8                                                   & $\mathbf{0}$                    &  $\mathbf{0}$    & $\mathbf{0}$    &  $\mathbf{0}$   &  $\mathbf{0}$   & 106   \\ \hline
esc32a                                      & 130                                                 & 93.8                                             & 129.2  & 109.2  & 104.6  &$\mathbf{83.1}$     & 433   \\ \hline
esc32b                                      & 168                                                 & 88.1                                             & 111.9  & 92.9   & 97.6   & $\mathbf{69.0}$    & 508   \\ \hline
esc32c                                      & 642                                                 & 7.8                                              & 15.6   & 14.0   & 15.0   & $\mathbf{4.0}$     & 552   \\ \hline
esc32d                                      & 200                                                 & 21                                               & 28.0   & 28.0   & 29.0   &$\mathbf{17.0}$    & 470   \\ \hline
esc32e                                      & 2                                                   & $\mathbf{0}$                             & $\mathbf{0}$    & $\mathbf{0}$    & $\mathbf{0}$    & $\mathbf{0}$     & 220   \\ \hline
esc32g                                      & 6                                                   & $\mathbf{0}$                             & 33.3   &$\mathbf{0}$    & $\mathbf{0}$    & $\mathbf{0}$      & 234   \\ \hline
esc32h                                      & 438                                                 & 18.3                                             & 25.1   & 19.6   & 25.1   & $\mathbf{13.2}$     & 570   \\ \hline
esc64a                                      & 116                                                 & 53.4                                             & 62.1   & 51.7   & 58.6   & $\mathbf{34.5}$    & 899   \\ \hline
esc128                                      & 64                                                  & $\mathbf{175}$                                               & 256.3  & 193.8  & 243.8  & 215.6   & 2045  \\ \hline
\end{tabular}
\caption{Comparison between upper bounds on the problems from the QAP library with (relatively) sparse $L$.}
\label{tb:qap}
\end{table}

\end{document}